\documentclass[reqno]{amsart}	
\usepackage{comment}
\usepackage{enumerate}
\usepackage{graphicx}  
\usepackage{nicefrac}
\usepackage{booktabs}
\usepackage{amsaddr, amsmath, amssymb,mathrsfs}
\usepackage{nicematrix}
\usepackage{stmaryrd}
\usepackage{caption}
\usepackage{subcaption}
\usepackage{algorithm}
\usepackage{algpseudocode}
\usepackage{algorithmicx}
\usepackage{tikz}
\usepackage{pgfplots}
\usepackage{geometry}
\usepackage{color}
\usepackage[onehalfspacing]{setspace}
\allowdisplaybreaks

\pgfplotsset{compat=1.15}
\definecolor{darkbrown}{rgb}{0.57, 0.40, 0.13}

\newcommand  {\F}   {\mathfrak{F}}
\newcommand  {\Q}   {\mathcal Q}
\newcommand  {\T}   {\mathscr T}

\newcommand  {\CC}  {\mathbb C}
\newcommand  {\HH}  {\mathbb H}
\newcommand  {\NN}  {\mathbb N}
\newcommand  {\PP}  {\mathbb P}
\newcommand  {\RR}  {\mathbb R}
\renewcommand{\SS}  {\mathbb S}
\newcommand  {\VV}  {\mathbb V}

\newcommand  {\Eig}  {\operatorname{Eig}}
\newcommand  {\dev}  {\operatorname{dev}}
\newcommand  {\diag} {\operatorname{diag}}
\newcommand  {\spann}  {\operatorname{span}}

\newcommand{\ul}[1]      {\underline{#1}}
\newcommand{\wh}[1]      {\widehat{#1}}
\newcommand{\tr}         {\operatorname{tr}}
\newcommand{\abs}[1]     {\left\vert #1 \right\vert}
\newcommand{\norm}[1]    {\left\Vert #1 \right\Vert}
\newcommand{\mv}[1]   {\mathfrak{#1}}          
\newcommand{\bs}[1]{\boldsymbol{#1}}     

\newcommand{\myspace}[1]{\mathbb{#1}}
\newcommand{\m}{\scalebox{0.70}[1.0]{$-$}}

\newtheorem{theorem}{Theorem}

\newtheorem{lemma}[theorem]      {Lemma}

\theoremstyle{definition}
\newtheorem{definition} {Definition}

\newtheorem*{example*}    {Example}
\newtheorem{remark}     {Remark}
\newtheorem*{remark*}	{Remark}

\begin{document}

\renewcommand{\figurename}{\small Figure}

\title[A Semismooth Newton Solver \& its application to Elastoplasticity]{A Semismooth Newton Solver and its application to \\ an $hp$-FE Discretization in Elastoplasticity}

\author[P.~Bammer, L.~Banz, M.~Sch\"{o}nauer \& A.~Schr\"{o}der]{Patrick Bammer$^{\, 1}$, Lothar Banz$^{\, 2}$, Miriam Sch\"{o}nauer$^{\, 2}$ \& Andreas Schr\"{o}der$^{\, 2}$}

\address{$^1$Mathematisches Institut, Universit\"{a}t Bern \\
Sidlerstr.~5, CH-3012 Bern, Switzerland \vspace{0.15cm} \\
$^2$Fachbereich Mathematik, Paris Lodron Universit\"{a}t Salzburg,\\ 
Hellbrunnerstr.~34, 5020 Salzburg, Austria}

\begin{abstract}
    In this paper, we consider a class of systems of nonlinear equations, which arise in discretized mixed formulations of problems in solid mechanics by $hp$-finite elements. We introduce a semismooth Newton solver for this specific class and prove its well-definedness and local convergence. Thereby, the analysis heavily relies on a special eigenvalue interplay of two matrices involved in the considered nonlinear system. Next, we apply the general results to an $hp$-finite element discretization of a problem in elastoplasticity, which can be formulated as a system of nonlinear equations of the above type by using biorthogonal basis functions. Finally, numerical examples demonstrate the applicability and robustness of the proposed semismooth Newton solver with respect to $h$ and $p$.
\end{abstract}


\thanks{%
    The authors gratefully acknowledge the support by the Bundesministerium f\"{u}r Frauen, Wissenschaft und Forschung (BMFWF) under the Sparkling Science project SPA 01-080 'MAJA -- Mathematische Algorithmen für Jedermann Analysiert'.
}


\maketitle
	

\section{Introduction}

    Elastoplasticity plays a key role in mechanical engineering, as it is used to model the behavior of construction materials \cite{ref:Poltronieri2014,ref:Wong2006}, and, in particular, elastoplasticity with kinematic hardening finds application in the modeling of metals; see,  e.g.,~\cite{ref:Josefson1995,ref:Meyer2018}. 
    The primal formulation of elastoplasticity is frequently described as a variational inequality of the second kind characterized by the appearance of a non-differentiable term $\psi$ -- the so-called \emph{plasticity functional}, see~\cite{ref:Chen1988,ref:Han2013}. The non-differentiability of $\psi$, however, causes many difficulties in the computation of a discrete solution.
    To circumvent these difficulties, one may reformulate the variational inequality as a mixed formulation in which a Lagrange multiplier is introduced to resolve the non-differentiable term $\psi$; see, e.g.,~\cite{ref:Han1995, ref:Schroeder2011}.
    In the discretization with $hp$-finite elements these constraints can be decoupled by using biorthogonal basis functions. This leads to a system of decoupled nonlinear equations which significantly simplifies the application of an efficient semismooth Newton solver; see \cite{ref:Bammer2022Icosahom}. 

    The structure of the resulting algebraic system of nonlinear equations is given by a nonlinear function $\F$, which is Lipschitz continuous and semismooth. The elements $\bs{H}$ of the Clarke subdifferential of $\F$ exhibit a certain block-matrix structure.
    Exploiting these block-matrices with regard to the requirements for a semismooth Newton method, we find that the convergence of the method heavily relies on the properties of the submatrices of $\bs{H}$; in particular, on some specific relations of the eigenvectors and eigenvalues between its submatrices, which we call \emph{eigencomplementarity}, see Definition \ref{subsec:Eigencomplement}. Moreover, it appears that these block structures as well as the property of eigencomplementarity is not restricted to problems of elastoplasticity. They also arise in other discretizations of structural mechanic problems, for instance, in frictional contact problems; see, e.g., \cite{ref:Banz2015Tresca, ref:Schroeder_PAMM2011, ref:Schroeder_MFE2011, ref:Schroeder_aPost2012, ref:Schroeder_contact2011}. Thus, the semismooth Newton method is analyzed for the whole class of systems of equations involving this block structure and the property of eigencomplementarity.
    
    In this paper, we consider a class of systems of nonlinear equations described by a nonlinear function $\mv{F}$ and analyze the convergence properties of a semismooth Newton solver applied to it. 
    To conclude the convergence of the semismooth Newton solver, we show that each element $\bs{H}$ of the Clarke subdifferential of $\F$ is regular.
    We accomplish this by exploiting the Schur complement properties and the eigencomplementarity of the submatrices of $\bs{H}$.
    With the regularity of $\bs{H}$ established, we apply the general results for semismooth Newton schemes in \cite{ref:Qi1993} to prove superlinear convergence and convergence order of $1+\alpha$ with $\alpha\in(0,1]$ depending on the smoothness of $\F$.
    Following the general convergence analysis, we consider a mixed variational formulation of a model problem of elastoplasticity with linearly kinematic hardening.
    We study how an $hp$-finite element discretization of this formulation can be restated as a system of nonlinear equations with the desired structure.
    We then prove that the submatrices of $\bs{H}$ are eigencomplementary and conclude the superlinear convergence of a semismooth Newton solver by using the general results.
    Finally, we conduct numerical experiments to showcase the robustness of the method with respect to $h$, $p$ and some projection parameter involved in the semismooth Newton method.
    
    The paper is structured as follows: In Section~\ref{sec: abstract setting} we introduce the general class of systems of nonlinear equations, to be analyzed in the following sections. Section~\ref{sec:semismooth_NS} is devoted to proving the superlinear/$1+\alpha$ order convergence of the semismooth Newton solver, which represents the main part of this paper. Then, in Section~\ref{sec:Application}, we examine the $hp$-finite element disretization of a model problem of elastoplasticity with linearly kinematic hardening and apply the general results to conclude the superlinear convergence of the semismooth Newton solver.
    Finally, the numerical examples in Section~\ref{sec:numeric} show the applicability of the semismooth Newton solver, the robustness of the number of iterations on the finite element spaces, and, most importantly, numerically validate the proven convergence rates. \vspace{0.15cm}
    
    \textbf{Notation} -- Throughout this article, we write $\ul{n}$ for the finite set $\lbrace 1,\ldots,n \rbrace$ where $n\in\NN$ is a positive integer. Furthermore, let
    \begin{align*}
        \SS_n := \Big\lbrace \bs{M}\in\RR^{n\times n} \; ;\;  \bs{M} =  \bs{M}^{\top}\Big\rbrace
    \end{align*}
    be the space of real symmetric $(n\times n)$-matrices and denote by $\bs{I}_{n\times n}$ the identity matrix of $\RR^{n\times n}$. Moreover, we write $\bs{M} \preccurlyeq 0$ and $\bs{M} \succcurlyeq 0$ to indicate that a matrix $\bs{M}\in\RR^{n\times n}$ is negative and positive semi-definite, respectively. Furthermore, for a matrix $\bs{M}$, we denote the spectrum of $\bs{M}$ by $\sigma(\bs{M})$ and write $\Eig_{\bs{M}}(\xi)$ for the eigenspace of an eigenvalue $\xi\in\sigma(\bs{M})$.
    Finally, the symbol $\bigoplus$ indicates the direct sum of vector spaces.


\smallskip

\section{Semismooth Newton Solver for a Specific Class of Nonlinear Equations}\label{sec: abstract setting}   


We start by specifying the class of systems of nonlinear equations for which the semismooth Newton solver is defined in the next subsection.


    \smallskip

    \subsection{Class of Nonlinear Equations}

    For $d\in\lbrace 2,3\rbrace$ and $L,M,N\in\NN$, let $\bs{A}\in\SS_{dM}$, $\bs{C}\in\SS_{LN}$ be positive definite matrices, and $\bs{D}\in\SS_{LN}$ a diagonal matrix with positive diagonal entries $(\bs{D})_{jj} = D_i > 0$ for $L(i-1) \leq j \leq Li$ and $i\in\ul{N}$. Furthermore, let $\bs{B}\in\RR^{dM\times LN}$ satisfy
    \begin{align*}
        \mv{a}^\top\bs{A}\,\mv{a}+ \mv{b}^\top\bs{C}\,\mv{b}+ 2\,\mv{a}^\top\bs{B}\,\mv{b} >0,
    \end{align*}
    for all $\mv{a}\in\RR^{dM}, \mv{b}\in\RR^{LN}$ with $(\mv{a},\mv{b})^{\top}\neq\mv{0}$, i.e., the block matrix 
    \begin{align}\label{eq: E Definition}
        \bs{E}:=\begin{pmatrix}
                \bs{A} & \bs{B}\\
                \bs{B}^\top & \bs{C}
            \end{pmatrix}
    \end{align}
    is positive definite.
    Then, for $\mv{l}\in\RR^{dM}$ and $K := dM+LN$ we define the affine linear mapping
    \begin{align}\label{def:lineartransform}
        \mv{L}:\RR^{dM} \times~\RR^{LN}\times~\RR^{LN}\to~\RR^K, \qquad
        \mv{L}(\mv{a},\mv{b},\mv{c}) 
         := \begin{pmatrix}
                \bs{A} & \bs{B} & \bs{0}\\
                \bs{B}^\top & \bs{C}& \bs{D}
            \end{pmatrix} \,
        \begin{pmatrix}
            \mv{a}\\
            \mv{b}\\
            \mv{c}
        \end{pmatrix} + \binom{\mv{l}}{\mv{0}}.
    \end{align}

    For $\mv{b} = (b_1,\ldots,b_{LN})^{\top}, \mv{c} = (c_1,\ldots,c_{LN})^{\top}\in\RR^{LN}$ we introduce the component vectors $\mv{b}_i$ and $\mv{c}_i$ by
    \begin{align*}
        \mv{b}_i :=(b_{L(i-1)+1},\ldots,b_{Li})^\top,
         \qquad 
        \mv{c}_i :=(c_{L(i-1)+1},\ldots,c_{Li})^\top,
         \qquad i\in \ul{N}
    \end{align*}
    and let
    \begin{align}\label{eq: Def: s_i}
        \mv{S}_i : \RR^L\times\RR^L \to \RR^L, \qquad
        \mv{S}_i(\mv{b}_i, \mv{c}_i)
         = \big( s_{i,1}(\mv{b}_i, \mv{c}_i),\ldots, s_{i,L}(\mv{b}_i, \mv{c}_i) \big)^{\top},\qquad i\in\ul{N}
    \end{align}
    be ($\alpha$-order with $\alpha\in(0,1]$) semismooth functions. Thereby, the ($\alpha$-order) semismoothness is to be understood in the sense of \cite{ref:Qi1993}.
    Therewith, we define the function
    \begin{align}\label{eq:semismooth_newtonF}
        \F : \RR^{dM}\times \RR^{LN}\times \RR^{LN} \to \RR^{K + LN}, \qquad
         \F(\mv{a},\mv{b},\mv{c})
          = \begin{pmatrix}
            \mv{L}(\mv{a},\mv{b},\mv{c}) \\
            \mv{S}_{1}(\mv{b}_1, \mv{c}_1) \\
            \vdots \\
            \mv{S}_{N}(\mv{b}_N, \mv{c}_N)
        \end{pmatrix},
    \end{align}
     which leads to the following (nonsmooth) problem: \emph{Find vectors $\mv{a}^\ast\in\RR^{dM}$, $\mv{b}^\ast\in\RR^{LN}$ and $\mv{c}^\ast\in\RR^{LN}$ such that:}
    \begin{align}\label{eq:semismooth_newtonF_02}
        \F (\mv{a}^\ast, \mv{b}^\ast, \mv{c}^\ast ) 
         = \mv{0}.
    \end{align}
    Finally, let $\partial \F$ denote the Clarke subdifferential of $\F$, i.e.,~$\partial \F(\mv{x},\mv{y},\mv{z})$ is the convex hull of all limits of the form 
    \begin{align*}
        \lim_{n\to\infty} \nabla \F\big(\mv{x}_n, \, \mv{y}_n, \, \mv{z}_n\big),
    \end{align*}
    where $(\mv{x}_n, \mv{y}_n, \mv{z}_n)^\top\in\RR^{dM}\times \RR^{LN}\times \RR^{LN}$ is a sequence of points in which $\F$ is differentiable that tends to $(\mv{x}, \mv{y}, \mv{z})^\top\in\RR^{dM}\times \RR^{LN}\times \RR^{LN}$ as $n\to\infty$. Here, $\nabla \F(\mv{x},\mv{y},\mv{z})$ represents the Jacobian matrix of $\F(\mv{x},\mv{y},\mv{z})$. The following lemma summarizes some basic properties of the nonlinear function $\F$.

        \begin{lemma}\label{lem:structure_subgradient_abstract}
            The function $\F$ is Lipschitz-continuous and ($\alpha$-order) semismooth. Furthermore, any element $\bs{H}\in\partial\F(\mv{a},\mv{b},\mv{c})$ for arbitrary $(\mv{a},\mv{b},\mv{c})^\top\in\RR^{dM}\times \RR^{LN}\times\RR^{LN}$ takes the form
                \begin{align}\label{eqn: H_structure}
                    \bs{H} = 
                    \begin{pNiceArray}{ccc}
                        \bs{A} & \bs{B} &\bs{0}   \\
                        \bs{B}^T& \bs{C} &\bs{D}  \\
                        \bs{0} & \bs{X} & \bs{Y}
                    \end{pNiceArray}
                \end{align}
                with block-diagonal matrices $\bs{X} = \diag(\bs{X}_1,\ldots,\bs{X}_N)$ and $\bs{Y} = \diag(\bs{Y}_1,\ldots,\bs{Y}_N)$, where $\bs{X}_i,\bs{Y}_i\in\RR^{L\times L}$ for $i\in\ul{N}$.

            %
        \end{lemma}
        \begin{proof}
            The Lipschitz-continuity and ($\alpha$-order) semismoothness of $\F$ follows directly from the Lipschitz-continuity and ($\alpha$-order) semismoothness of $\mv{L}$ and $\mv{S}_i(\mv{b}_i,\mv{c}_i)$.
            Through basic calculations with respect to $\mv{L}$ we obtain the first two block rows in \eqref{eqn: H_structure}.
            Now, we consider the partial derivatives of the components $\mv{S}_i(\mv{b}_i,\mv{c}_i)$ for $i\in\ul{N}$ in \eqref{eq:semismooth_newtonF}. First, there holds
            \begin{align*}
                \frac{\partial}{\partial a_j} s_{i,l}(\mv{b}_i,\mv{c}_i) = 0 \qquad
            \forall \, (i,l,j)\in \ul{N} \times \ul{L} \times \ul{dM}.
            \end{align*}
            Secondly, for the derivatives of $s_{i,l}(\mv{b}_i,\mv{c}_i)$ with respect to $b_j$ and $c_j$ for $j\in\ul{LN}$ we obtain
            \begin{align*}
                \frac{\partial}{\partial b_j} s_{i,l}(\mv{b}_i,\mv{c}_i) 
                = \frac{\partial}{\partial c_j} s_{i,l}(\mv{b}_i,\mv{c}_i)
                = 0
            \end{align*}
            for all  $i\in\ul{N}$ and all $j\in\ul{LN}$ with $j\notin\lbrace L(i-1)+1, \ldots, Li \rbrace$. This leads to the zero matrix and the block diagonal structure of the matrices $\bs{X}$ and $\bs{Y}$ in the last row of \eqref{eqn: H_structure}.  
        \end{proof}
       We use the following semismooth Newton solver for the system \eqref{eq:semismooth_newtonF_02} given by:
        \begin{algorithm}[H]
        \caption{Semismooth Newton Solver}\label{eq:SSN_iterates}
            \begin{algorithmic}[1]
                \Procedure{NewtonSolver} {$\mv{a}^{(0)},\mv{b}^{(0)},\mv{c}^{(0)}, tol$}
                    \State $k \gets 0$
                    \While{$|\F\big(\mv{a}^{(k)},\mv{b}^{(k)},\mv{c}^{(k)}\big)| > tol$}
                        \State Choose $\bs{X}_k, \bs{Y}_k\in\RR^{LN\times LN}$, cf. \eqref{eqn: H_structure}, such that $\bs{H}_k\in\partial\F\big( \mv{a}^{(k)}, \mv{b}^{(k)}, \mv{c}^{(k)}\big)$
                        \State Solve $\bs{H}_k \big( \Delta\mv{a}^{(k)}, \Delta\mv{b}^{(k)}, \Delta\mv{c}^{(k)} \big)^{\top}
                        = - \F\big( \mv{a}^{(k)}, \mv{b}^{(k)}, \mv{c}^{(k)} \big)$
                        \State Set $\big( \mv{a}^{(k+1)}, \mv{b}^{(k+1)}, \mv{c}^{(k+1)} \big)^{\top}
                        = \big( \mv{a}^{(k)}, \mv{b}^{(k)}, \mv{c}^{(k)} \big)^{\top} +\, \big( \Delta\mv{a}^{(k)}, \Delta\mv{b}^{(k)}, \Delta\mv{c}^{(k)} \big)^{\top}$
                        \State $k\gets k+1$
                    \EndWhile
                    \State \textbf{return} $\big(\mv{a}^{(k)},\mv{b}^{(k)},\mv{c}^{(k)}\big)$
                \EndProcedure
            \end{algorithmic}
        \end{algorithm}
        \begin{remark}
                To improve the convergence of the semismooth Newton solver one may introduce a step length parameter $t_k\in(0,1]$ in line 6 of Algorithm~\ref{eq:SSN_iterates}, which results in
                \begin{align*}
                    \big( \mv{a}^{(k+1)}, \mv{b}^{(k+1)}, \mv{c}^{(k+1)} \big)^{\top}
                                = \big( \mv{a}^{(k)}, \mv{b}^{(k)}, \mv{c}^{(k)} \big)^{\top} +t_k\, \big( \Delta\mv{a}^{(k)}, \Delta\mv{b}^{(k)}, \Delta\mv{c}^{(k)} \big)^{\top}.
                \end{align*}
                Note that the step length parameter $t_k$ has to be chosen by an adequate step length selection procedure, see, e.g., \cite{DeLuca1996, ref:Qi1998}.
                Furthermore, we note that by the linearity of $\mv{L}$ we have 
                $$
                \mv{L}\big(\mv{a}^{(k)},\mv{b}^{(k)},\mv{c}^{(k)}\big)=\mv{o}\qquad \forall\,k\geq k_0
                $$ 
                if $\mv{L}\big(\mv{a}^{(k_0)},\mv{b}^{(k_0)},\mv{c}^{(k_0)}\big)=\mv{o}$ for some $k_0\in\NN$. That condition definitely holds true if a full semismooth Newton step with $t_k=1$ has been performed.
        \end{remark}
        


    \smallskip

    \subsection{Eigencomplementarity}\label{subsec:eigencomp}

    The convergence analysis of the semismooth Newton Solver \ref{eq:SSN_iterates} fundamentally relies on a certain interplay of the eigenvalues and eigenvectors of the matrices $\bs{X}_i$ and $\bs{Y}_i$ in Lemma~\ref{lem:structure_subgradient_abstract}. Therefore, we introduce the following definition.

    \begin{definition}\label{subsec:Eigencomplement}   
        A pair $(\bs{F},\bs{G})$ of symmetric matrices $\bs{F},\bs{G}\in\SS_L$, with $L\in\NN$, is called \textit{eigencomplementary} if
        \begin{itemize}
            \item[$\diamond$] $\bs{F}\preccurlyeq 0$,
            \item[$\diamond$] $\bs{G}\succcurlyeq 0$,
            \item[$\diamond$] they have an eigenbasis $\mv{B}$ in common,
        \end{itemize} 
        and in the case that $\bs{F}$ and $\bs{G}$ are both singular we additionally assume
                \begin{align}\label{eq: Eigencomp both singular}
                    \bigoplus\limits_{\substack{ \xi \, \in \, \sigma(\bs{F}) \\ \xi \, <  \, 0 }} \Eig_{\bs{F}}(\xi) = \Eig_{\bs{G}}(0).
                \end{align}

    \end{definition}

    \begin{remark}\label{re: Eigencomp}
        \begin{enumerate}
            \item[\bf (a)] Note that Definition~\ref{subsec:Eigencomplement} is exclusively formulated for symmetric matrices. This guarantees that all eigenvalues are real.
            \item[\bf (b)] For the ease of notation, Definition~\ref{subsec:Eigencomplement} is formulated for a pair of matrices. The order, however, is not of great importance, but to formulate the following results it is advantageous to specify which matrix is negative semi-definite and which one is positive definite.
            \item[\bf (c)] Since the matrices $\bs{F},\bs{G}\in\SS_L$ of an eigencomplementary pair $(\bs{F},\bs{G})$ have a basis $\mv{B}$ of $L$ pairwise different eigenvectors $\mv{v}_1,\dots,\mv{v}_L$ in common, we have $\bs{F}\,\mv{v}_j = \xi\,\mv{v}_j$ and $\bs{G}\,\mv{v}_j = \eta\,\mv{v}_j$ for all $j\in\ul{L}$ with some eigenvalues $\xi\in\sigma(\bs{F})$ and $\eta\in\sigma(\bs{G})$ with $\xi\leq 0 \leq \eta$. Additionally, as the matrices $\bs{F}$ and $\bs{G}$ are symmetric, the common basis $\mv{B}$ of eigenvectors can be chosen to form an orthogonal basis of $\RR^L$. Note that the matrices have an eigenbasis $\mv{B}$ in common if and only if they are simultaneously diagonalizable, i.e., there exists an orthogonal matrix $\bs{S}$ such that 
           \begin{align}\label{eq: simult diag}
                \bs{F} = \bs{S} \, \bs{D}_{\bs{F}} \, \bs{S}^\top,\qquad
                \bs{G} = \bs{S} \, \bs{D}_{\bs{G}} \, \bs{S}^\top, 
           \end{align}
           where the columns of $\bs{S}$ correspond to the (orthogonal) eigenvectors of $\bs{F}$ and $\bs{G}$.
           Here, $\bs{D}_{\bs{F}}, \bs{D}_{\bs{G}}\in\RR^{L\times L}$ are diagonal matrices with diagonal entries corresponding to the eigenvalues of $\bs{F}$ and $\bs{G}$, respectively.
           \item[\bf (d)] In the case that $\bs{F}$ and $\bs{G}$ are both singular, the condition \eqref{eq: Eigencomp both singular} means that $\bs{D}_F\bs{D}_G =\bs{0}$ and, additionally, it holds that either $ (\bs{D}_F)_{jj}\neq 0$ or $(\bs{D}_G)_{jj}\neq 0$ for each $j\in\ul{L}$ with $\bs{D}_F, \bs{D}_G$ from \eqref{eq: simult diag}. This complementarity relation between the eigenvalues of $\bs{F}$ and $\bs{G}$ motivates the name \emph{eigencomplementarity}.
           \item[\bf (e)] Furthermore, condition \eqref{eq: Eigencomp both singular} implies that if $\bs{F}$ and $\bs{G}$ are both singular neither of them can be the zero matrix.
           \item[\bf (f)] We note that condition \eqref{eq: Eigencomp both singular} can equivalently be expressed as follows
           \begin{align*}
               \bigoplus\limits_{\substack{ \eta \, \in \, \sigma(\bs{G}) \\ \eta \, >  \, 0 }} \Eig_{\bs{G}}(\eta) = \Eig_{\bs{F}}(0).
           \end{align*}
  
        \end{enumerate}
        
    \end{remark}

          To elucidate the definition of eigencomplementary matrices, we give an example of two singular eigencomplementary matrices. For the singular symmetric matrices 
          \begin{align*}
                    \bs{F}=\begin{pmatrix}
                        \m 3 & 1 & 1 & \m 3\\
                        1 & \m 3 & \m 3 & 1\\
                        1 & \m 3 & \m 3 & 1\\
                        \m 3 & 1 & 1 & \m3 \\
                    \end{pmatrix}
                    ,\quad \bs{G}=\begin{pmatrix}
                        10 & 4 & \m 4 & \m 10\\
                        4 & 10 & \m 10 & \m 4\\
                        \m 4 & \m 10 & 10 & 4\\
                        \m 10 & \m 4 & 4 & 10\\
                    \end{pmatrix}
           \end{align*}
          the vectors
            \begin{align*}
                \mv{v}_1 =\begin{pmatrix}
                    \m 1\\
                    1\\
                    \m 1 \\
                    1\\
                \end{pmatrix}
                , \quad \mv{v}_2= \begin{pmatrix}
                    1\\
                    1\\
                    1\\
                    1
                \end{pmatrix}
                , \quad \mv{v}_3 = \begin{pmatrix}
                    1\\
                    \m 1\\
                    \m 1\\
                    1
                \end{pmatrix}
                , \quad \mv{v}_4 = \begin{pmatrix}
                    \m 1\\
                    \m 1\\
                    1\\
                    1
                \end{pmatrix}
            \end{align*}
            form a common eigenbasis. Moreover, we have $\sigma(\bs{F}) = \lbrace 0, \m4, \m8\rbrace$ and $\sigma(\bs{G}) = \lbrace 0, 12, 28\rbrace$. Since
            \begin{align*}
                \Eig_{\bs{F}}(\m4)\boldsymbol{\oplus}\Eig_{\bs{F}}(\m8) = \spann(\mv{v}_2)\boldsymbol{\oplus} \spann(\mv{v}_3)
                 =\Eig_{\bs{G}}(0)
            \end{align*}
            the matrices $\bs{F}$ and $\bs{G}$ are eigencomplementary.
        \begin{lemma}\label{lem: XY neg semidef}
            Let $(\bs{F}, \bs{G})$ be a pair of eigencomplementary matrices $\bs{F},\bs{G}\in\SS_L$. Then it holds that
            \begin{enumerate}
                \item[\bf 1.] If the matrix $\bs{F}$ is regular, then $\bs{F}^{\m1}\bs{G}\preccurlyeq 0$.
                \item[\bf 2.] If the matrix $\bs{G}$ is regular, then $\bs{F} \, \bs{G}^{\m1} \preccurlyeq 0$.
            \end{enumerate}
        \end{lemma}
        \begin{proof}
            By the eigencomplementarity of $(\bs{F}, \bs{G})$ there exist an orthogonal matrix $\bs{S}$ and diagonal matrices $\bs{D}_F,\bs{D}_G\in\RR^{L\times L}$ such as in \eqref{eq: simult diag}.
            If $\bs{F}$ is regular it holds that
           \begin{align*}
                \bs{F}^{\m1} = \bs{S} \, \bs{D}_{\bs{F}}^{\m1} \, \bs{S}^\top,
           \end{align*}
           which implies that
           \begin{align*}
                \bs{F}^{\m1}\bs{G} = \bs{S} \, \bs{D}_{\bs{F}}^{\m1} \, \bs{D}_{\bs{G}} \, \bs{S}^\top.
           \end{align*}
           Thus, the entries of the diagonal matrix $\bs{D}_{\bs{F}}^{\m1} \, \bs{D}_{\bs{G}}$ correspond to the eigenvalues of $\bs{F}^{\m1}\bs{G}$ and are of the form $\xi^{\m1}\, \eta$, where $\xi\in\sigma(\bs{F})$ and $\eta\in\sigma(\bs{G})$. As in this case $\bs{F}$ is negative definite and $\bs{G}$ is positive semi-definite, we conclude that $\xi^{\m1} \, \eta\leq 0$, which proves the assertion. The case for regular $\bs{G}$ follows the same arguments.
        \end{proof}
        \begin{lemma}\label{lem: Singular Eigencomp Orthogonal}
            Let $(\bs{F}, \bs{G})$ be a pair of singular, eigencomplementary matrices $\bs{F},\bs{G}\in\SS_L$. If $\bs{G}\,\mv{u} = \bs{F}\,\mv{w}$ for two vectors $\mv{u},\mv{w}\in\RR^L$, then it follows that $\mv{u}^\top\mv{w}=0$.
        \end{lemma}
        \begin{proof}
            By the eigencomplementarity of $(\bs{F}, \bs{G})$ there exist an orthogonal matrix $\bs{S}$ and diagonal matrices $\bs{D}_F,\bs{D}_G\in\RR^{L\times L}$ such as in \eqref{eq: simult diag}, cf. Remark~\ref{re: Eigencomp}\textbf{(c)}. We write $\mv{u},\mv{w}\in\RR^L$
            \begin{align*}
                \mv{u}
                = \sum_{j\in\ul{L}} \gamma_{j} \, \mv{v}_j = \bs{S}\,\mv{\gamma}, \quad
                \mv{w} 
                = \sum_{j\in\ul{L}} \beta_{j} \, \mv{v}_j = \bs{S}\,\mv{\beta},
            \end{align*}
            where $\mv{\gamma}=(\gamma_1,\dots,\gamma_L)^\top,\,\mv{\beta}=(\beta_1,\dots,\beta_L)^\top\in\RR^L$ are the coordinate vectors of $\mv{u}$ and $\mv{w}$, respectively. Thus, it holds that
            \begin{align*}
                \bs{G}\,\mv{u}&=(\bs{S}\bs{D}_G\bs{S}^\top)\bs{S}\,\mv{\gamma} = \bs{S}\,\bs{D}_G\,\mv{\gamma} \quad \text{ and }\quad \bs{F}\,\mv{w} = (\bs{S}\bs{D}_F\bs{S}^\top)\bs{S}\,\mv{\beta} =\bs{S}\,\bs{D}_F\,\mv{\beta},
            \end{align*}
            which implies that
            \begin{align}\label{eq: DG g = DF b}
               \bs{D}_G\,\mv{\gamma}=\bs{D}_F\,\mv{\beta}
            \end{align}
            since $\bs{G}\,\mv{u} = \bs{F}\,\mv{w}$.
             As $\bs{G}$ is singular there exists an index set $J\subset\ul{L}$ with $J\neq\emptyset$ such that $(\bs{D}_G)_{jj} = 0$ for all $j\in J$ and $(\bs{D}_G)_{jj}\neq 0$ for $j\in\ul{L}\setminus J$. By \eqref{eq: Eigencomp both singular} and Remark~\ref{re: Eigencomp}{\bf (d)} it follows that $(\bs{D}_F)_{jj}\neq 0$ for $j\in J$ and $(\bs{D}_F)_{jj} = 0$ for $j\in \ul{L}\setminus J$, which, in view of \eqref{eq: DG g = DF b}, implies that
            \begin{align*}
                \gamma_{j} = 0 \quad\forall\, j\in J\qquad\text{and} \qquad
                \beta_{j} = 0 \quad\forall\, j\in \ul{L}\setminus J.
            \end{align*}
            Consequently, from the orthogonality of $\bs{S}$ it follows that
            \begin{align*}
                \mv{u}^\top\mv{w}=\mv{\gamma}^\top\bs{S}^\top\bs{S}\,\mv{\beta} = \mv{\gamma}^\top\mv{\beta}= \sum\limits_{j=1}^L\gamma_j\,\beta_j=0,
            \end{align*}
            which completes the argument.
        \end{proof}
        


    \smallskip

    \section{Well-Definedness \& Convergence of the Semismooth Newton Solver}\label{sec:semismooth_NS}

        In this section, we prove the well-definedness and local convergence of the semismooth Newton solver given in Algorithm \ref{eq:SSN_iterates}. For this purpose, we assume the existence of a solution $(\mv{a}^{\ast}, \mv{b}^{\ast}, \mv{c}^{\ast} )$ of \eqref{eq:semismooth_newtonF_02}.
        
        For each $\bs{H}\in\partial\F (\mv{a}^{\ast}, \mv{b}^{\ast}, \mv{c}^{\ast})$, which is of the structure \eqref{eqn: H_structure}, we consider the block matrix $\bs{E}$ of the form
        \begin{align*}
            \bs{E}=\begin{pmatrix}
                \bs{A} & \bs{B}\\
                \bs{B}^\top & \bs{C}
            \end{pmatrix}
        \end{align*}
        as defined in \eqref{eq: E Definition}.
        Then, the Schur complement of $\bs{E}$ with respect to $\bs{A}$ is given by
        \begin{align*}
            \bs{S_E} := \bs{C}-\bs{B}^\top\bs{A}^{\m 1}\bs{B}.
        \end{align*}
        Note that $\bs{S_E}$ is symmetric and positive definite by the positive definiteness of $\bs{A}$ and $\bs{E}$, see \cite{ref:Zhang2005}. Thus, the Schur complement of $\bs{H}$ with respect to $\bs{E}$ results in
        \begin{align*}
            \bs{S_H}:=\bs{Y}-
            \begin{pmatrix}
                \bs{0} &\bs{X}
            \end{pmatrix}\:
            \bs{E}^{\m 1}
            \begin{pmatrix}
                \bs{0}\\
                \bs{D}
            \end{pmatrix}
            = \bs{Y} - \bs{X} \, \bs{S_E}^{\m 1} \bs{D}.
        \end{align*}
        %

        %
        
        %

            %
        \begin{lemma}\label{prop:SHregular}
            Let the pair $(\bs{X}_i,\bs{Y}_i)$ be eigencomplementary for all $i\in\ul{N}$. Then, the Schur complement $\bs{S_H}$ of $\bs{H}$ with respect to $\bs{E}$ is regular.
        \end{lemma}

        \begin{proof}

            Let us assume that there is an $\mv{h}= (h_1,\ldots,h_{LN})^\top\in(\ker\bs{S_H})\setminus \{\mv{o}\}$ and let  $\mv{w} = (w_1,\ldots, w_{LN})^{\top}\in \RR^{LN}$ be determined by
               \begin{align*}
                   \mv{w} := \bs{S_E}^{\m 1} \bs{D} \, \mv{h}.
               \end{align*}
            For any $i\in\ul{N}$, we define the quantities
            \begin{align*}
                \mv{h}_i
                := ( h_{L(i-1)+1},\dots, h_{Li})^{\top} \in \RR^L, \qquad
                \mv{w}_i
                := (w_{L(i-1)+1}, \ldots, w_{Li} )^{\top} \in\RR^L.
            \end{align*}
             Note that since $\bs{X}$ and $\bs{Y}$ are block-diagonal, cf.~Lemma \ref{lem:structure_subgradient}, the equation
            \begin{align}\label{eq: Darstellung S_H h}
                \bs{S_H} \, \mv{h} 
                 = \big( \bs{Y} - \bs{X} \, \bs{S_E}^{\m 1} \bs{D} \big) \, \mv{h}
                 = \mv{o}
            \end{align} 
            turns out to be equivalent to
            \begin{align}\label{eq:SchurComp kernel}
                \bs{Y}_i \, \mv{h}_i
                 = \bs{X}_i \, \mv{w}_i \qquad \forall \,  i\in\underline{N}.
            \end{align}
            Finally, we denote by $\bs{D}_i$ the corresponding blocks of $\bs{D}$, i.e., $\bs{D}_i = D_i \, \bs{I}_{L\times L}$ for $i\in\ul{N}$.
            
           We now show that 
            \begin{align}\label{eq: hDw leq 0}
                \mv{h}_i^\top \, \bs{D}_i \, \mv{w}_i
                 \leq 0 \qquad \forall \, i\in\underline{N}
            \end{align}
            holds true.
            For this purpose, let $i\in\ul{N}$ be arbitrary but fixed. We distinguish three cases: \vspace{0.15cm}

                    \emph{\textbf{1}.} Let $\bs{Y}_i$ be regular. Then, the matrix $\bs{X}_i\bs{Y}_i^{\m 1}$ is negative semi-definite by Lemma \ref{lem: XY neg semidef} and, therefore, we can conclude from \eqref{eq:SchurComp kernel} and the symmetry of the matrices that
                    \begin{align*}
                        \mv{h}^\top_i \, \bs{D}_i \, \mv{w}_i
                         = D_i\,(\bs{Y}_i^{\m 1}\bs{X}_i\,\mv{w}_i)^\top\, \mv{w}_i = D_i\,\mv{w}_i^\top\,\bs{X}_i\bs{Y}_i^{\m 1}\, \mv{w}_i \leq 0.
                    \end{align*}
                   \emph{\textbf{2.}} Let $\bs{X}_i$ be regular. Then, the matrix $\bs{X}_i^{\m 1}\bs{Y}_i$ is negative semi-definite by Lemma \ref{lem: XY neg semidef} and, therefore, we can conclude from \eqref{eq:SchurComp kernel} that
                    \begin{align*}
                        \mv{h}^\top_i \, \bs{D}_i \, \mv{w}_i
                         = D_i \, \mv{h}^\top \, \mv{w}_i
                         = D_i \, \mv{h}^\top_i \, \bs{X}_i^{\m 1} \bs{Y}_i \, \mv{h}_i
                         \leq 0.
                    \end{align*}
                    \emph{\textbf{3.}} Let both $\bs{X}_i$ and $\bs{Y}_i$ be singular. Then, from \eqref{eq:SchurComp kernel} and Lemma~\ref{lem: Singular Eigencomp Orthogonal} it follows that $\mv{h}_i^\top\mv{w}_i=0$, which leads to 
                    \begin{align*}
                        0 = D_i \, \mv{h}_i^\top \, \mv{w}_i
                         = D_i \, \mv{h}_i^\top \, \bs{I}_{L\times L} \, \mv{w}_i
                         = \mv{h}_i^\top \, \bs{D}_i \, \mv{w}_i.
                    \end{align*}
                Thus, we have shown \eqref{eq: hDw leq 0}.
                The positive definiteness of $\bs{S_E}$ in combination with the definition of $\mv{w}$, then leads to the contradiction
                \begin{align*}
                    0 < ( \bs{D} \, \mv{h} )^\top \, \bs{S_E}^{-1} ( \bs{D} \, \mv{h})
                     = \sum_{i\in\ul{N}} \mv{h}_i^\top \, \bs{D}_i \, \mv{w}_i
                     \leq 0.
                \end{align*}
                Consequently, it holds that $\ker\bs{S_H} = \{\mv{o}\}$ which shows that $\bs{S_H}$ is regular.
        \end{proof}
        %
        %
        \begin{theorem}\label{thm:conditions_convergence}
            If the pair $(\bs{X}_i,\bs{Y}_i)$ is eigencomplementary for all $i\in\ul{N}$, then Algorithm~\ref{eq:SSN_iterates} is well-defined in a neighborhood of the solution $(\mv{a}^{\ast}, \mv{b}^{\ast}, \mv{c}^{\ast} )$ of \eqref{eq:semismooth_newtonF_02} and converges superlinearly. Moreover, if the involved functions $\mv{S}_i$, cf.~\eqref{eq: Def: s_i}, are $\alpha$-order semismooth with $\alpha\in(0,1]$, then the convergence is of order $1+\alpha$.
        \end{theorem}
        \begin{proof}
            The well-definedness as well as the ($1+\alpha$)-order convergence of the semismooth Newton solver in Algorithm~\ref{eq:SSN_iterates} follow from \cite[Thm.~3.2]{ref:Qi1993}. Thereby, two preliminaries, besides the existance of a solution $(\mv{a}^{\ast}, \mv{b}^{\ast}, \mv{c}^{\ast} )$, have to be fulfilled:
            \begin{enumerate}
                \item $\F$ is locally Lipschitz-continuous and ($\alpha$-order) semismooth at the solution $(\mv{a}^{\ast}, \mv{b}^{\ast}, \mv{c}^{\ast} )$.
                \item All $\bs{H} \in \partial\F(\mv{a}^{\ast}, \mv{b}^{\ast}, \mv{c}^{\ast})$ are regular.
            \end{enumerate}
            By Lemma~\ref{lem:structure_subgradient_abstract}, $\F$ is Lipschitz-continuous and ($\alpha$-order) semismooth. Furthermore, the regularity of all $\bs{H}\in\partial\F(\mv{a}^{\ast}, \mv{b}^{\ast}, \mv{c}^{\ast})$ follows from Lemma~\ref{prop:SHregular} as the regularity of $\bs{S}_{\bs{H}}$ implies the regularity of $\bs{H}$.
        \end{proof}
        \begin{remark}
            Note that the results of Section~\ref{sec:semismooth_NS} still apply if there exist $i\in\ul{N}$ such that the pair $(\bs{Y}_i,\bs{X}_i)$ is eigencomplementary instead of $(\bs{X}_i,\bs{Y}_i)$, cf. Remark~\ref{re: Eigencomp} on the ordering of the matrices in the definition of eigencomplementarity.
        \end{remark}


    \smallskip

    \section{Application to an Elastoplastic Problem}\label{sec:Application}
    In this section, we consider a model problem of elastoplasticity with linearly kinematic hardening and its mixed variational formulation. We refer to \cite{ref:Bammer2022Icosahom, ref:Bammer2023Apriori, ref:Bammer2023Posteriori} for more details on the subject.


    \smallskip

        \subsection{The model problem}\label{sec:Elastoplasticity with Hardening}

        Let the reference configuration of an elastoplastic body be represented by a bounded polygonal domain $\Omega\subset\RR^d$, $d\in\lbrace 2,3\rbrace$, with Lipschitz-boundary $\Gamma := \partial\Omega$ and outer unit normal $\mv{n}$. Moreover, let the body be clamped at the Dirichlet-boundary part $\Gamma_D\subseteq\Gamma$, which we assume to be a closed set of positive surface measure, and set $\Gamma_N := \Gamma\setminus \Gamma_D$. Then, for given volume force $\mv{f}:\Omega\to\RR^d$ and surface traction $\mv{g}:\Omega\to\RR^d$ the model problem of elastoplasticity with linearly kinematic hardening, see, e.g., \cite{ref:Han2013}, is to find a vector-valued displacement field $\mv{u}:\Omega\to \RR^d$ and a matrix-valued plastic strain $\bs{p}:\Omega\to\SS_{d,0}$, with
        \begin{align*}
            \SS_{d,0} := \bigg\lbrace \bs{q} \in \SS_d  \; ; \,  \tr(\bs{q}) := \sum_{i\in\ul{d}} q_{ii} = 0 \bigg\rbrace,
        \end{align*}
        such that
        \begin{align}\label{eq:model_problem}
            \begin{split}
                -\operatorname{div} \bs{\sigma}(\mv{u},\bs{p}) &= \mv{f}   \; \qquad \;  \text{ in } \Omega, \\
                \mv{u} &= \mv{o}                                           \;  \qquad \,  \text{ on } \Gamma_D, \\
                \bs{\sigma}(\mv{u},\bs{p}) \, \mv{n} &= \mv{g}  \; \qquad \, \text{ on } \Gamma_N, \\
                \bs{\sigma}(\mv{u},\bs{p})-\HH \, \bs{p} &\in \partial j(\bs{p})  \quad   \text{in } \Omega.
            \end{split}
        \end{align}
        The stress tensor $\bs{\sigma}(\cdot)$ is given by the constitutive equation $\bs{\sigma}(\mv{u}) := \CC \, \big( \bs{\varepsilon}(\mv{u}) - \bs{p} \big)$ with the fourth-order elasticity tensor $\CC$ and the strain tensor $\bs{\varepsilon}(\cdot)$, which is given by the relation $\bs{\varepsilon}(\mv{u}) := \frac{1}{2} \, \big( \nabla\mv{u} + (\nabla\mv{u})^{\top} \big)$. Furthermore,  we denote by $\HH$ the fourth-order hardening tensor. We suppose that the elasticity tensor $\CC$ and the hardening tensor $\HH$ have the usual symmetry properties and are uniformly elliptic and uniformly bounded. Finally, $\partial j(\cdot)$ represents the subdifferential of the dissipation function $j(\cdot)$, which is defined as $j(\bs{q}) := \sigma_y \, \abs{\bs{q}}_F$ for $\bs{q}\in\SS_{d,0}$. Here, $\sigma_y$ is the yield stress in uniaxial tension that is assumed to be a positive constant and $\abs{ \, \cdot \, }_F$ is the Frobenius norm, which is induced by the Frobenius inner product $\bs{p} : \bs{q} := \sum_{i,j\in\ul{d}} p_{ij} \, q_{ij}$ for $\bs{p}=(p_{ij}), \bs{q}=(q_{ij})\in\RR^{d\times  d}$.


    \smallskip

    \subsection{Weak formulation}

        Let $(\cdot,\cdot)_{0,\Omega}$ denote the $L^2(\Omega, X)$ inner product for $X\in\lbrace \RR, \RR^d, \RR^{d \times d}\rbrace$, inducing the $L^2$-norm $ \norm{\, \cdot \, }_{0, \Omega}$, respectively, and define the vector spaces
        \begin{align*}
            V := \big\lbrace \mv{v} \in H^1(\Omega, \RR^d) \; ; \; \mv{v}_{\, | \, \Gamma_D} = \mv{o} \big\rbrace, \qquad
            Q := L^2(\Omega,\SS_{d,0}).
        \end{align*}
        Thereby, $H^1(\Omega,\RR^d)$ is the space of vector-valued functions $\mv{v}=(v_1,\ldots,v_d)^{\top}$ with components $v_i\in H^1(\Omega)$
        and $\mv{v}_{\, | \, \Gamma_D}$ denotes the trace of $\mv{v}$ on the boundary part $\Gamma_D$. Note that $\VV := V\times Q$ is a Hilbert space, which can be equipped with the norm
        \begin{align*}
            \norm{(\mv{v},\bs{q})} ^2
            := \norm{\mv{v}}_{1,\Omega}^2 + \norm{\bs{q}}_{0,\Omega}^2 ,
        \end{align*}
        where $\norm{\mv{v}}_{1,\Omega}^2 := \norm{\mv{v}}_{0,\Omega}^2 + \abs{\mv{v}}_{1,\Omega}^2$ and $\abs{\mv{v}}_{1,\Omega}^2 := \big( \bs{\varepsilon}(\mv{v}), \bs{\varepsilon}(\mv{v}) \big)_{0,\Omega}$. Due to Korn's inequality, the norm $\norm{ \, \cdot \,}_{1,\Omega}$ is equivalent to the usual Sobolev norm on $V$. Furthermore, we denote by $V^{\ast}$ the dual space of $V$ endowed with the dual norm $\norm{ \, \cdot \,}_{V^{\ast}}$ and write $\langle\cdot,\cdot\rangle$ for the duality pairing between $V$ and $V^{\ast}$. Finally, let $H^{-1/2}(\Gamma_N,\RR^d)$ be the dual space of the trace space of $V$ restricted to $\Gamma_N$ and write $\langle\cdot,\cdot\rangle_{\Gamma_N}$ for the duality pairing.

        We define the bilinear form $a:\VV\times \VV\to \RR$ as
        \begin{equation*}
            a\big( (\mv{u},\bs{p}), (\mv{v},\bs{q}) \big)
            := \big( \bs{\sigma}(\mv{u}, \bs{p}), \bs{\varepsilon}(\mv{v}) - \bs{q} \big)_{0,\Omega} + (\HH \, \bs{p}, \bs{q})_{0,\Omega},
        \end{equation*}
        the plasticity functional $\psi: Q\to \RR$ as
        \begin{equation*}
            \psi(\bs{q})
            := (\sigma_y, \abs{\bs{q}}_F)_{0,\Omega},
        \end{equation*}
        and the linear form $\ell: V\to \RR$ as
        \begin{align*}
            \ell(\mv{v}):= \langle \mv{f}, \mv{v}\rangle + \langle \mv{g}, \mv{v}\rangle_{\Gamma_N},
        \end{align*}
         where $\mv{f}\in V^{\ast}$ and $\mv{g}\in H^{-1/2}(\Gamma_N,\RR^d)$. Therewith, we consider a well-established weak formulation of the model problem \eqref{eq:model_problem}, see, e.g.,~\cite{ref:Han2013}: \emph{Find a pair $(\mv{u},\bs{p})\in \VV$ such that the variational inequality of the second kind
        \begin{align}\label{eq:variq_second_kind}
            a\big( (\mv{u},\bs{p}), (\mv{v}-\mv{u},\bs{q}-\bs{p}) \big) + \psi(\bs{q}) - \psi(\bs{p}) \geq \ell(\mv{v}-\mv{u}) \qquad 
            \forall \, (\mv{v},\bs{q})\in \VV
        \end{align}
        holds true.} Due to the properties of the elasticity tensor $\CC$ and hardening tensor $\HH$, the bilinear form $a(\cdot,\cdot)$ is symmetric, continuous, and $\VV$-elliptic, cf.~\cite{ref:Han2013}. The plasticity functional $\psi(\cdot)$ is convex, Lipschitz-continuous, and subdifferentiable, see~\cite{ref:Bammer2023Apriori}, and the linear form $\ell(\cdot)$ is continuous for $\mv{f}\in V^{\ast}$ and $\mv{g}\in H^{-1/2}(\Omega,\RR^d)$.  With these assumptions, there exists a unique solution $(\mv{u},\bs{p})\in \VV$ of \eqref{eq:variq_second_kind}, cf.~\cite{ref:Han2013}.

        For a mixed variational formulation of \eqref{eq:model_problem}, we introduce the non-empty, closed, and convex set 
        \begin{align*}
            \Lambda 
            := \big\lbrace \bs{\mu}\in Q \; ; \; \abs{ \bs{\mu} }_F \leq \sigma_y \text{ a.e.~in }\Omega \big\rbrace
        \end{align*}
        of admissible Lagrange multipliers that can equivalently be represented as
        \begin{align*}
            \Lambda
            = \big\lbrace \bs{\mu}\in Q \; ; \; (\bs{\mu}, \bs{q})_{0,\Omega} \leq \psi(\bs{q}) \text{ for all } \bs{q}\in Q \big\rbrace,
        \end{align*}
        cf.~\cite{ref:Bammer2023Apriori}. Thereby, the Frobenius norm in the definition of the plasticity functional $\psi(\cdot)$ can be resolved by a Lagrange multiplier as $\psi(\bs{q}) = \sup_{\bs{\mu}\in\Lambda} (\bs{\mu}, \bs{q})_{0,\Omega}$ for any $\bs{q}\in Q$. Then, a mixed variational formulation of \eqref{eq:model_problem} is given by: \emph{Find a triple $(\mv{u},\bs{p},\bs{\lambda})\in \VV \times \Lambda$ such that}
        \begin{align}\label{eq:mixed_variationalF}
            \begin{split}
                a\big( (\mv{u},\bs{p}), (\mv{v},\bs{q}) \big) + (\bs{\lambda}, \bs{q})_{0,\Omega} 
                &= \ell(\mv{v})
                 \qquad \forall \, (\mv{v},\bs{q})\in \VV, \\
                (\bs{\mu}-\bs{\lambda}, \bs{p})_{0,\Omega} 
                &\leq 0
                 \quad\qquad \; \forall \, \bs{\mu}\in\Lambda.
            \end{split}
        \end{align}
        In fact, the weak formulations \eqref{eq:variq_second_kind} and \eqref{eq:mixed_variationalF} are equivalent in the sense that if $(\mv{u},\bs{p})\in \VV$ solves \eqref{eq:variq_second_kind}, then $(\mv{u},\bs{p},\bs{\lambda})$ with
        \begin{align}\label{eq:repres_continuousLambda}
            \bs{\lambda} 
            = \operatorname{dev}\big( \bs{\sigma}(\mv{u},\bs{p})-\HH \, \bs{p} \big)
        \end{align}
        is a solution to \eqref{eq:mixed_variationalF} and, conversely, if $(\mv{u},\bs{p},\bs{\lambda})\in \VV \times \Lambda$ solves \eqref{eq:mixed_variationalF}, then $(\mv{u},\bs{p})$ is a solution of \eqref{eq:variq_second_kind} and the identity \eqref{eq:repres_continuousLambda} is true; see \cite{ref:Bammer2023Apriori} for a proof. Here, $\operatorname{dev}\bs{\tau}:= \bs{\tau} - \frac{1}{d}\,\operatorname{tr}(\bs{\tau})\,\bs{I}_{d\times d}$ denotes the deviatoric part of a matrix $\bs{\tau}\in\RR^{d\times d}$.


    \smallskip
    
    \subsection{$hp$-Finite Element Discretization}\label{sec:discretization}

    For an $hp$-finite element discretization of the weak formulations \eqref{eq:variq_second_kind} and \eqref{eq:mixed_variationalF}, respectively, let $\T_{h}$ be a locally quasi-uniform finite element mesh of $\Omega$ consisting of convex and shape regular quadrilaterals or hexahedrons. Moreover, we choose $\wh{T}:=[\m 1,1]^d$ as reference element and let $\mv{M}_T:\wh{T}\to T$ denote the bi/trilinear bijective mapping for $T\in\T_h$. We set $h := (h_T)_{T\in\T_h}$ and $p := (p_T)_{T\in\T_h}$ where $h_T$ and $p_T$ represent the local element size and the polynomial degree, respectively. We assume that the local polynomial degrees of neighboring elements are comparable in the sense of \cite{ref:Melenk2005}. For the discretization of the displacement field $\mv{u}$ and the plastic strain $\bs{p}$, we use the conforming $hp$-finite element spaces
    \begin{align*}
        V_{hp} &:= \Big\lbrace \mv{v}_{hp}\in V \; ; \; \mv{v}_{hp \, | \, T}\circ \mv{M}_T\in \big(\PP_{p_T}(\wh{T})\big)^d \text{ for all } T\in\T_h\Big\rbrace, \\
        Q_{hp} &:= \Big\lbrace \bs{q}_{hp}\in Q \; ; \; \bs{q}_{hp \, | \, T}\circ \mv{M}_T\in \big(\PP_{p_T-1}(\wh{T})\big)^{d\times d} \text{ for all } T\in\T_h\Big\rbrace,
    \end{align*}
    where $\PP_{p_T}(\wh{T})$ is the space of polynomials up to degree $p_T$ on the reference element $\wh{T}$ and set $\VV_{hp} := V_{hp}\times Q_{hp}$. To obtain a discretization of the set $\Lambda$, we need a computable approximation $\psi_{hp}(\cdot)$ of the plasticity functional $\psi(\cdot)$. To that end, let $\hat{\mv{x}}_{k,T}\in \widehat{T}$ be the tensor product Gauss quadrature points on $\widehat{T}$ for $k \in \ul{n_T}$, where $n_T := p_T^d$ for $T\in\T_h$. The corresponding positive weights are denoted by $\hat{\omega}_{k,T}\in\RR$. Let us introduce the mesh dependent quadrature rule 
    \begin{align*}
        \Q_{hp}(\cdot)
        := \sum_{T\in\T_h} \Q_{hp,T}(\cdot),
    \end{align*}
    where the local quantities $\Q_{hp, T}(\cdot)$ are given by
    \begin{align*}
        \Q_{hp, T}(f) :=
        \begin{cases}
            |T| \, f\big(\mv{M}_T(\mv{0})\big), & \text{if } p_T = 1, \\
            \sum_{k=1}^{n_T} \hat{\omega}_{k,T} \, |\det \nabla \mv{M}_T(\hat{\mv{x}}_{k,T})| \; f\big(\mv{M}_T(\hat{\mv{x}}_{k,T})\big), & \text{if } p_T \geq 2,
        \end{cases}
        \qquad T\in\T_h
    \end{align*}
    as introduced in \cite{ref:Bammer2023Apriori}.
    Therewith, we define the discrete plasticity functional $\psi_{hp}: Q_{hp}\to\RR$ by
    \begin{align*}
        \psi_{hp}(\bs{q}_{hp}) 
        := \Q_{hp}\big( \sigma_y \, \abs{\bs{q}_{hp}}_F \big).
    \end{align*}
    We note that $\psi_{hp}(\cdot)$ results from a nodal interpolation of the Frobenius norm appearing in the definition of $\psi(\cdot)$ and can be exactly evaluated by the mesh dependent quadrature rule $\Q_{hp}(\cdot)$, see \cite{ref:Bammer2022Icosahom, ref:Bammer2023Apriori,ref:Gwinner2009}.   
    A discretization of the variational inequality \eqref{eq:variq_second_kind} is given by: \emph{Find a pair $(\mv{u}_{hp},\bs{p}_{hp})\in\VV_{hp}$ such that}
    \begin{align}\label{eq:discrete_variq_second_kind}
        a\big( (\mv{u}_{hp},\bs{p}_{hp}), (\mv{v}_{hp}-\mv{u}_{hp}, \bs{q}_{hp}-\bs{p}_{hp}) \big) + \psi_{hp}(\bs{q}_{hp}) - \psi_{hp}(\bs{p}_{hp}) 
        \geq \ell(\mv{v}_{hp} - \mv{u}_{hp})
        \qquad \forall \, (\mv{v}_{hp},\bs{q}_{hp})\in \VV_{hp}.
    \end{align}
    To obtain a discretization of the mixed variational formulation \eqref{eq:mixed_variationalF} we introduce the non-empty, convex, and closed set of admissible discrete Lagrange multipliers by
    \begin{align*}
        \Lambda_{hp} := \Big\lbrace \bs{\mu}_{hp}\in Q_{hp} \; ; \; (\bs{\mu}_{hp},\bs{q}_{hp})_{0,\Omega} \leq \psi_{hp}(\bs{q}_{hp}) \text{ for all } \bs{q}_{hp}\in Q_{hp}  \Big\rbrace,
    \end{align*}
    which leads to the discrete mixed formulation: \emph{Find a triple $(\mv{u}_{hp},\bs{p}_{hp},\bs{\lambda}_{hp})\in\VV_{hp}\times \Lambda_{hp}$ such that}
    \begin{subequations}\label{eq:discrete_mixed_variationalF}
        \begin{alignat}{2}
            a\big( (\mv{u}_{hp},\bs{p}_{hp}),(\mv{v}_{hp},\bs{q}_{hp}) \big) + (\bs{\lambda}_{hp}, \bs{q}_{hp})_{0,\Omega} &= \ell(\mv{v}_{hp}) & \qquad & \forall \, (\mv{v}_{hp},\bs{q}_{hp}) \in \VV_{hp},\label{eq:discrete_mixed_variationalF_01}\\
            (\bs{\mu}_{hp}-\bs{\lambda}_{hp}, \bs{p}_{hp})_{0,\Omega}
            &\leq 0 & \quad & \forall \, \bs{\mu}_{hp} \in \Lambda_{hp}.\label{eq:discrete_mixed_variationalF_02}
        \end{alignat}
    \end{subequations}
    In \cite{ref:Bammer2023Posteriori} it is shown that the two discretizations \eqref{eq:discrete_variq_second_kind} and \eqref{eq:discrete_mixed_variationalF} are equivalent in the sense that if $(\mv{u}_{hp},\bs{p}_{hp})\in \VV_{hp}$ solves \eqref{eq:discrete_variq_second_kind}, then $(\mv{u}_{hp},\bs{p}_{hp},\bs{\lambda}_{hp})$ with
    \begin{align}\label{eq:repres_discreteLambda}
        \bs{\lambda}_{hp}
        = \mathcal{P}_{hp}\big( \dev (\bs{\sigma}(\mv{u}_{hp},\bs{p}_{hp})-\HH \, \bs{p}_{hp}) \big),
    \end{align}
    is a solution to \eqref{eq:discrete_mixed_variationalF}, where $\mathcal{P}_{hp} : Q\to Q_{hp}$ denotes the $L^2$-projection operator. Conversely, if $(\mv{u}_{hp},\bs{p}_{hp},\bs{\lambda}_{hp})\in \VV_{hp} \times \Lambda_{hp}$ solves \eqref{eq:discrete_mixed_variationalF}, then $(\mv{u}_{hp},\bs{p}_{hp})$ is a solution of \eqref{eq:discrete_variq_second_kind} and the identity \eqref{eq:repres_discreteLambda} holds. Moreover, it is shown in \cite{ref:Bammer2023Posteriori} that there exists a unique solution to the discrete variational inequality \eqref{eq:discrete_variq_second_kind}. Due to the equivalence of the two discrete formulations, there consequently exists a solution $(\mv{u}_{hp},\bs{p}_{hp},\bs{\lambda}_{hp})\in \VV_{hp} \times \Lambda_{hp}$ to the discrete mixed formulation \eqref{eq:discrete_mixed_variationalF}, where the first two components are unique. In \cite{ref:Bammer2023Apriori} the uniqueness of the discrete Lagrange multiplier $\bs{\lambda}_{hp}$, norm-convergence and a priori error estimates are shown, under the assumption
    \begin{align}\label{eq:requirement_detF}
        \det \nabla \mv{M}_T \in \PP_1(\wh{T}) \qquad 
        \forall \, T\in\T_h \text{ with } p_T\geq 2.
    \end{align}
    Note that \eqref{eq:requirement_detF} only slightly limits the shape of the mesh elements in the case that $d=3$ and $p_T > 1$. Without assumption \eqref{eq:requirement_detF}, the convergence of higher-order finite element schemes based on \eqref{eq:discrete_variq_second_kind} is an open problem.


    \smallskip
    
    \subsection{Algebraic Representation of the Mixed Problem}\label{sec:system_decopled_NE}

    In order to apply the Newton Solver of Algorithm~\ref{eq:SSN_iterates}, we rewrite the discrete mixed problem \eqref{eq:discrete_mixed_variationalF} in terms of a decoupled system of nonlinear equations of the form \eqref{eq:semismooth_newtonF}. Therefore, we first decouple the constraints in $\Lambda_{hp}$ and \eqref{eq:discrete_mixed_variationalF_02}. For this purpose, let $\{\wh{\phi}_{k,T}\}_{k\in\ul{n_T}}$ be the Lagrange basis functions on $\wh{T}$ defined via the Gauss points $\hat{\mv{x}}_{l,T}$, i.e.,
    \begin{align*}
        \wh{\phi}_{k,T}\in\PP_{p_T-1}(\wh{T}), \quad
        \wh{\phi}_{k,T}(\hat{\mv{x}}_{l,T}) = \delta_{kl} \qquad
        \forall \, k,l\in\ul{n_T} \quad \forall \, T\in\T_h,
    \end{align*}
    where $\delta_{kl}$ is the usual Kronecker delta symbol. Moreover, let $\phi_1,\ldots,\phi_N$ be piecewise defined as
    \begin{align*}
        \phi_{\zeta(k,T')\, | \, T} 
        := 
        \begin{cases}
                \wh{\phi}_{k,T'} \circ \mv{M}_{T'}^{-1},& \text{if } T=T',\\
                0,& \text{if } T\neq T', \\
        \end{cases} 
        \qquad  T,T'\in\T_h, \quad k\in\ul{n_T},
    \end{align*}
    where $\zeta:\big\lbrace (k,T) \; ; \;  T\in\mathcal{T}_h, \, k\in\ul{n_T} \big\rbrace\rightarrow \lbrace 1,\ldots,N \rbrace$ with $N :=\sum_{T\in\mathcal{T}_h} n_T$ is a one-to-one numbering. Finally, let $\varphi_1,\ldots,\varphi_N$ be the biorthogonal functions to $\phi_1,\ldots,\phi_N$, see, e.g.,~\cite{ref:Lamichhane2007biorthogonal}, which are uniquely determined by the conditions $\varphi_{\zeta(k,T)\, | \, T} \circ\mv{M}_T\in\PP_{p_T-1}(\wh{T})$ for $T\in\T_h$, $k\in\ul{n_T}$ and 
    \begin{align*}
        (\phi_i,\varphi_j)_{0,\Omega} 
        = \delta_{ij} \, (\phi_i, 1)_{0,\Omega} \qquad 
        \forall\, i,j\in\ul{N}.
    \end{align*}
    Under assumption \eqref{eq:requirement_detF} we have $\varphi_{\zeta(k,T)} = \phi_{\zeta(k,T)}$ for $k\in\ul{n_T}$ and $T\in\T_h$ and in any case it holds that
    \begin{align*}
        Q_{hp} 
        = \bigg\lbrace \sum_{i\in\ul{N}} \bs{q}_i \, \phi_i \; ; \; \bs{q}_i\in\SS_{d,0} \, \text{ for all } \, i\in\ul{N} \bigg\rbrace
        = \bigg\lbrace \sum_{i\in\ul{N}} \bs{\mu}_i \, \varphi_i \; ; \; \bs{\mu}_i\in\SS_{d,0} \, \text{ for all } \, i\in\ul{N}\bigg\rbrace,
    \end{align*}
    see~\cite{ref:Bammer2023Posteriori}. Therewith, and by defining the quantities $D_{i} := (\phi_i, 1)_{0,\Omega}>0$ and $\sigma_i := D_{i}^{-1} (\sigma_y, \phi_i)_{0,\Omega}$ for all $i\in\ul{N}$ we obtain 
    \begin{align*}
        \Lambda_{hp} 
        = \bigg\lbrace \sum_{i\in\ul{N}} \bs{\mu}_i \, \varphi_i \; ; \;  \bs{\mu}_i \in \SS_{d,0} \text{ and } \abs{ \bs{\mu}_i }_F \leq \sigma_i \, \text{ for all } \, i\in\ul{N} \bigg\rbrace,
    \end{align*}
    see \cite{ref:Bammer2022Icosahom} for a proof. Representing $\bs{\lambda}_{hp}\in \Lambda_{hp}$ and $\bs{p}_{hp}\in Q_{hp}$ as
    \begin{align*}
        \bs{\lambda}_{hp} = \sum_{i\in\ul{N}} \bs{\lambda}_i \, \varphi_i, \qquad
        \bs{p}_{hp} = \sum_{i\in\ul{N}} \bs{p}_i \, \phi_i, 
    \end{align*}
    it is shown in \cite{ref:Bammer2022Icosahom} that $\bs{\lambda}_{hp}$ satisfies the inequality \eqref{eq:discrete_mixed_variationalF_02} if and only if there holds $\bs{\lambda}_i : \bs{p}_i = \sigma_i \, \abs{ \bs{p}_i }_F$ for all $i\in\ul{N}$, which, in fact, decouples the constraints in $\Lambda_{hp}$ and \eqref{eq:discrete_mixed_variationalF_02}. If $(\mv{u}_{hp},\bs{p}_{hp},\bs{\lambda}_{hp})\in\VV_{hp}\times\Lambda_{hp}$ is the unique solution of \eqref{eq:discrete_mixed_variationalF} it is pointed out in \cite{ref:Bammer2022Icosahom} that the following implications hold true
    \begin{align*}
        \abs{ \bs{\lambda}_i }_F < \sigma_i 
        \;\Longrightarrow \; \bs{p}_i = \bs{0} , \qquad
        \abs{ \bs{\lambda}_i }_F = \sigma_i 
        \; \Longrightarrow \; \exists\, c \geq 0 \text{ with } \bs{p}_i = c \, \bs{\lambda}_i. 
    \end{align*}
    This suggests to introduce the nonlinear functions $\bs{\chi}_i : \SS_{d,0}\longrightarrow\SS_{d,0}$ given by
    \begin{align}\label{eq:defi_chi_i}
        \bs{\chi}_i(\bs{q}_i, \bs{\mu}_i) 
        := \max \big\lbrace \sigma_i, \abs{ \bs{\mu}_i + \varrho \, \bs{q}_i }_F \big\rbrace \, \bs{\mu}_i - \sigma_i \, (\bs{\mu}_i + \varrho \, \bs{q}_i) \qquad
        \forall \, i\in\ul{N},
    \end{align} 
    for some $\varrho > 0$, where 
    \begin{align}\label{eq: q_hp m_uhp definition}
        \bs{q}_{hp} = \sum_{i\in\ul{N}} \bs{q}_i \, \phi_i,\qquad
        \bs{\mu}_{hp} = \sum_{i\in\ul{N}} \bs{\mu}_i \, \varphi_i.
    \end{align}
    Therewith, the discrete Lagrange multiplier $\bs{\lambda}_{hp}\in\Lambda_{hp}$ satisfies the condition \eqref{eq:discrete_mixed_variationalF_02} if and only if there holds
    \begin{align}\label{eq: xi 0 condition}
        \bs{\chi}_i(\bs{p}_i, \bs{\lambda}_i) 
        = \bs{0} \qquad
        \forall \, i\in\ul{N},
    \end{align}
    see~\cite{ref:Bammer2022Icosahom} for a proof. This allows rewriting \eqref{eq:discrete_mixed_variationalF} as a system of nonlinear equations.    
       To this end, we first choose an appropriate basis of the space $\SS_{d,0}$. If $d = 2$, we take 
        \begin{align*}
            \bs{\Phi}_1 := \frac{1}{\sqrt{2}} \,  \begin{pmatrix} 1 & 0 \\ 0 & \m 1 \end{pmatrix}\; \text{ and }\;\bs{\Phi}_2 := \frac{1}{\sqrt{2}} \, \begin{pmatrix} 0 & 1 \\ 1 & 0 \end{pmatrix},
        \end{align*}
        and for $d=3$ we take \pagebreak
        \begin{align*}
            \bs{\Phi}_1 
            := \frac{1}{\sqrt{2}} & \, \begin{pmatrix} 1 & 0 & 0 \\ 0 & \m 1 & 0 \\ 0 & 0 & 0 \end{pmatrix}, \quad
            \bs{\Phi}_2 
            := \frac{1}{\sqrt{6}} \, \begin{pmatrix} 1 & 0 & 0 \\ 0 & 1 & 0 \\ 0 & 0 & \m 2 \end{pmatrix}, \quad
            \bs{\Phi}_3 
            := \frac{1}{\sqrt{2}} \, \begin{pmatrix} 0 & 1 & 0 \\ 1 & 0 & 0 \\ 0 & 0 & 0 \end{pmatrix}, \\
            & \quad \bs{\Phi}_4 
            := \frac{1}{\sqrt{2}} \, \begin{pmatrix} 0 & 0 & 1 \\ 0 & 0 & 0 \\ 1 & 0 & 0 \end{pmatrix}, \quad
            \bs{\Phi}_5 
            := \frac{1}{\sqrt{2}} \, \begin{pmatrix} 0 & 0 & 0 \\ 0 & 0 & 1 \\ 0 & 1 & 0 \end{pmatrix}. 
            \end{align*}
        Note that in both cases the basis functions are orthonormal with respect to the Frobenius inner product. Thus, we can represent $\bs{q}_{hp}\in Q_{hp}$ and $\bs{\mu}_{hp}\in\Lambda_{hp}$ by
        \begin{align}\label{eq:coeffvec_1}
            \bs{q}_{hp}
            = \sum_{i\in\ul{N}} \sum_{k\in\ul{L}} q_{L(i-1)+k} \, \bs{\Phi}_k \, \phi_i, \qquad
            \bs{\mu}_{hp}
            = \sum_{i\in\ul{N}} \sum_{k\in\ul{L}} \mu_{L(i-1)+k} \, \bs{\Phi}_k \, \varphi_i,
        \end{align}
        where $L := \frac{1}{2} \, (d-1) \, (d+2)$. Furthermore, let $\vartheta_1,\ldots,\vartheta_M$ be chosen such that $\lbrace \mv{e}_k \, \vartheta_i \; ; \; k\in\ul{d} \text{ and } i\in\ul{M} \rbrace$ forms a basis of $V_{hp}$, where $\mv{e}_k$ denotes the $k$-th Euclidean unit vector in $\RR^d$. Then, we can represent any $\mv{v}_{hp}\in V_{hp}$ by
        \begin{align}\label{eq: v_hp definition}
            \mv{v}_{hp} 
            = \sum_{i\in\ul{M}} \sum_{k\in\ul{d}} v_{d(i-1)+k} \, \mv{e}_k \, \vartheta_i.
        \end{align}
        Thus, an element $(\mv{v}_{hp},\bs{q}_{hp},\bs{\mu}_{hp})\in\VV_{hp}\times \Lambda_{hp}$ is completely represented by \eqref{eq: v_hp definition} and \eqref{eq: q_hp m_uhp definition} with the coefficient vectors
        \begin{align*}
            \mv{a}
            := (v_1,\ldots,v_{dM})^{\top}\in\RR^{dM}, \qquad
            \mv{b}
            := (q_1,\ldots,q_{LN})^{\top}\in\RR^{LN}, \qquad
            \mv{c} 
            := (\mu_1,\ldots,\mu_{LN})^{\top}\in\RR^{LN}.
        \end{align*}
        In particular, we indicate the coefficient vectors corresponding to the discrete solution $(\mv{u}_{hp},\bs{p}_{hp},\bs{\lambda}_{hp})\in\VV_{hp}\times\Lambda_{hp}$ of \eqref{eq:discrete_mixed_variationalF} by $\mv{a}^\ast,\mv{b}^\ast$ and $\mv{c}^\ast$, respectively. 
        Furthermore, we write
        \begin{align}\label{eq: Chi Basisdarstellung}
            \bs{\chi}_i(\bs{q}_i, \bs{\mu}_i) 
            = \sum_{k\in\ul{L}} \chi_{i,k} \, \bs{\Phi}_k \qquad
            \forall \, i\in\ul{N}.
        \end{align}
        Note that the $l$-th coefficient $\chi_{i,l}=\chi_{i,l}(\bs{q}_i, \bs{\mu}_i)$ in \eqref{eq: Chi Basisdarstellung} can be obtained by taking the Frobenius inner product of $\bs{\chi}_i(\bs{q}_i, \bs{\mu}_i)$ of \eqref{eq:defi_chi_i} with the basis function $\bs{\Phi}_l$. Recalling \eqref{eq:coeffvec_1}, we have
        \begin{align}
            \chi_{i,l} 
            &= \Bigg( \sum_{k\in\ul{L}} \max \bigg\lbrace \sigma_i, \Big( \sum_{k\in\ul{L}} ( \mu_{L(i-1)+k} + \varrho \, q_{L(i-1) + k} )^2 \Big)^{1/2} \bigg\rbrace \, \mu_{L(i-1)+k} \, \bs{\Phi}_k \nonumber \\ 
            & \qquad\qquad - \sigma_i \, ( \mu_{L(i-1)+k} + \varrho \, q_{L(i-1) + k} ) \, \bs{\Phi}_k \Bigg) : \bs{\Phi}_l \nonumber \\
            &= \max \big\lbrace \sigma_i, \abs{ \mv{c}_{ i} + \varrho \, \mv{b}_{ i} } \big\rbrace \, \mu_{L(i-1)+l} - \sigma_i \, \big( \mu_{L(i-1)+l} + \varrho \, q_{L(i-1)+l} \big), \label{eq:component_chi}
        \end{align}
        where $\abs{\, \cdot\,}$ denotes the Euclidean norm in $\RR^L$ and the vectors $ \mv{c}_{ i}, \mv{b}_{ i}\in\RR^L$ are given by
        \begin{align*}
            \mv{c}_{ i} 
            := (\mu_{L(i-1)+1}, \ldots,\mu_{Li})^{\top}, \qquad
            \mv{b}_{ i}
            := (q_{L(i-1)+1}, \ldots,q_{Li})^{\top}.
        \end{align*}
        By defining the semismooth functions $\mv{S}_{i}(\mv{b}_{i},\mv{c}_{i}) := (\chi_{i,1},\ldots,\chi_{i,L})^{\top}\in\RR^{L}$ for $i\in\ul{N}$, see e.g.~\cite{Hintermuller2002, Hueber2008}, we introduce the function  
        \begin{align}\label{eq:semismooth_newtonF_Elastoplast}
            \F:\RR^{dM}\times \RR^{LN}\times\RR^{LN}\to \RR^{dM+LN+LN},\qquad \F (\mv{a}, \mv{b}, \mv{c} ) 
            := 
            \begin{pmatrix}
                \mv{L}(\mv{a}, \mv{b}, \mv{c})\\
                \mv{S}_{1}(\mv{b}_{1},\mv{c}_{1}) \\
                \vdots \\
                \mv{S}_{N}(\mv{b}_{N},\mv{c}_{N})
            \end{pmatrix},
        \end{align}
        where $\mv{L}$ is the affine linear mapping of the form \eqref{def:lineartransform} with the symmetric, positive definite matrices $\bs{A}\in\RR^{dM\times dM}$ and $\bs{C}\in\RR^{LN\times LN}$, the positive definite diagonal matrix $\bs{D}\in\RR^{LN\times LN}$, and coupling matrix $\bs{B}\in\RR^{dM\times LN}$, which are given componentwise by \pagebreak
        \begin{alignat*}{2}
            A_{d(i-1)+k, \, d(j-1)+l} 
            &:= a\big( (\mv{e}_l \, \vartheta_j,\mathbf{0}) , (\mv{e}_k \, \vartheta_i,\mathbf{0}) \big)
            \qquad &&\forall \, i,j\in\ul{M} \quad \forall \, k,l\in\ul{d}, \\
            C_{L(i-1)+k, \, L(j-1)+l} 
            &:= a\big( (\mv{0},\boldsymbol{\Phi}_l \, \phi_i) , (\mv{0},\boldsymbol{\Phi}_k \, \phi_j) \big)
            \qquad &&\forall \, i,j\in\ul{N} \quad \forall \, k,l\in\ul{L}, \\
            D_{L(i-1)+k, \, L(j-1)+l} 
            &:= \delta_{lk} \, \delta_{ij} \, D_{i}
            \qquad &&\forall \, i,j\in\ul{N} \quad \forall \, k,l\in\ul{L}, \\
            B_{d(i-1)+k, \, L(j-1)+l} 
            &:= -a\big( (\mv{o},\bs{\Phi}_l \, \phi_j) , (\mv{e}_k \, \vartheta_i, \mathbf{0}) \big)
            \qquad &&\forall \, i\in\ul{N} \quad \forall \, j\in\ul{M} \quad \forall \, k\in\ul{d} \quad \forall \, l \in \ul{L}.
        \end{alignat*}
        The vector $\mv{l}\in\RR^{dM}$ in $\mv{L}$ has the components $l_{d(i-1)+k} := -\ell(\mv{e}_k \, \vartheta_i)$ for $i\in\ul{M}$ and $k\in\ul{d}$. Finally, the discrete mixed formulation \eqref{eq:discrete_mixed_variationalF} can equivalently be represented as the following system of nonlinear equations: \emph{Find the coefficient vectors $\mv{a}^\ast\in\RR^{dM}$, $\mv{b}^\ast\in\RR^{LN}$ and $ \mv{c}^\ast\in\RR^{LN}$ such that}
        \begin{align}\label{eq:semismooth_newtonF_Elastoplast_02}
            \F (\mv{a}^\ast, \mv{b}^\ast, \mv{c}^\ast ) 
            = \mv{0}.
        \end{align}
        By the unique existence of a solution of \eqref{eq:discrete_mixed_variationalF}, we conclude that the problem \eqref{eq:semismooth_newtonF_Elastoplast_02} has a unique solution $(\mv{a}^*,\mv{b}^*,\mv{c}^*)^\top\in\RR^{dM}\times \RR^{LN}\times \RR^{LN}$ as well and takes the form \eqref{eq:semismooth_newtonF_02}. 
        By Lemma~\ref{lem:structure_subgradient_abstract} the function $\F$ is Lipschitz-continuous and semismooth and any element $\bs{H}\in\partial\F(\mv{a},\mv{b},\mv{c})$ is of the form \eqref{eqn: H_structure}, i.e.,
                \begin{align}
                    \bs{H} = 
                    \begin{pNiceArray}{ccc}
                        \bs{A} & \bs{B} &\bs{0}   \\
                        \bs{B}^T& \bs{C} &\bs{D}  \\
                        \bs{0} & \bs{X} & \bs{Y}
                    \end{pNiceArray}
                \end{align}
                with block-diagonal matrices $\bs{X} = \diag(\bs{X}_1,\ldots,\bs{X}_N)$ and $\bs{Y} = \diag(\bs{Y}_1,\ldots,\bs{Y}_N)$.

    
    \smallskip
        
    \subsection{Superlinear Convergence of a Semismooth Newton Solver}\label{sec:ElastoSuperLinear}
    
        In the following, we denote by $\mv{x}\otimes\mv{y}$ the dyadic product of two vectors $\mv{x},\mv{y}\in\RR^L$ and, for $i\in\ul{N}$ and $\varrho > 0$, we define
        \begin{align*}
            \mv{v}_{\varrho,i} := \mv{c}_{i} + \varrho \, \mv{b}_{i}.
        \end{align*}
       \begin{lemma}\label{lem:structure_subgradient}
           
            For $i\in\ul{N}$, the matrices $\bs{X}_i,\bs{Y}_i\in\RR^{L\times L}$ are given by
            \begin{align}\label{eq:desiredStruc_Xi_Yi}
                \bs{X}_i
                = \tau_i\, \varrho \, \mv{c}_{i} \otimes \frac{\mv{v}_{\varrho,i}}{|\mv{v}_{\varrho,i}|} - \varrho \, \sigma_i \, \bs{I}_{L\times L}, \qquad
                \bs{Y}_i
                = \tau_i\left(\mv{c}_{i} \otimes  \frac{\mv{v}_{\varrho,i}}{|\mv{v}_{\varrho,i}|} + \big(|\mv{v}_{\varrho,i}|-\sigma_i\big) \, \bs{I}_{L\times L}\right)
            \end{align}
            for some $\tau_i\in[0,1]$. Specifically, in the cases $|\mv{v}_{\varrho,i}|<\sigma_i$ and $|\mv{v}_{\varrho,i}|>\sigma_i$ we have $\tau_i=0$ and $\tau_i = 1$, respectively.
        \end{lemma}
        \begin{proof}
            The matrix $\bs{E}$ as defined in \eqref{eq: E Definition} is symmetric and positive definite, as $\bs{E}$ corresponds to the bilinear form $a(\cdot,\cdot)$. Let $(\mv{a},\mv{b},\mv{c})\in\RR^{dM}\times\RR^{LN}\times \RR^{LN}$ be arbitrary.
            For the structure of the matrices $\bs{X}_i$ and $\bs{Y}_i$, for $i\in\ul{N}$, we have to consider the components according to $\mv{S}_i(\mv{b}_i,\mv{c}_i)$ in \eqref{eq:semismooth_newtonF_Elastoplast}. Note that for any $i\in\ul{N}$ the function $\mv{S}_i(\mv{b}_i, \mv{c}_i)$ is differentiable as long as there either holds $|\mv{v}_{\varrho,i}|>\sigma_i$ or $|\mv{v}_{\varrho,i}|<\sigma_i$. We therefore distinguish several cases:
            \begin{enumerate}
                \item[\bf 1.] If for all $i\in\ul{N}$ there either holds $|\mv{v}_{\varrho,i}|>\sigma_i$ or $|\mv{v}_{\varrho,i}|<\sigma_i$, i.e., $\F$ is differentiable in $(\mv{a},\mv{b},\mv{c})$, we compute the partial derivatives of the component $\chi_{i,l}$:
                \begin{itemize}
                    \item[--] If $|\mv{v}_{\varrho,i}|>\sigma_i$ we obtain
                    \begin{align*}
                        \frac{\partial}{\partial q_j} \chi_{i,l} 
                        &= \varrho \, \frac{\mu_{L(i-1)+l} \, (\mu_j + \varrho \, q_j)}{ |\mv{v}_{\varrho,i}|} - \varrho \, \sigma_i \, \delta_{j,L(i-1)+l}, \\
                        \frac{\partial}{\partial \mu_j} \chi_{i,l} 
                        &= \frac{\mu_{L(i-1)+l} \, (\mu_j + \varrho \, q_j)}{ \abs{\mv{v}_{\varrho,i}} } + \big(|\mv{v}_{\varrho,i}|-\sigma_i\big) \, \delta_{j,L(i-1)+l}
                    \end{align*}
                    for all $i\in\ul{N}$ and $L(i-1)+1\leq j \leq Li$, which results in
                    \begin{align}\label{eq:strucure_Xi_Yi_C1}
                        \bs{X}_i
                        = \varrho \, \mv{c}_{i} \otimes \frac{\mv{v}_{\varrho,i}}{|\mv{v}_{\varrho,i}|} - \varrho \, \sigma_i \, \bs{I}_{L\times L}, \qquad
                        \bs{Y}_i
                        = \mv{c}_{i} \otimes  \frac{\mv{v}_{\varrho,i}}{|\mv{v}_{\varrho,i}|} + \big(|\mv{v}_{\varrho,i}|-\sigma_i\big) \, \bs{I}_{L\times L}.
                    \end{align}
                    Note that \eqref{eq:strucure_Xi_Yi_C1} is of the structure \eqref{eq:desiredStruc_Xi_Yi} with $\tau_i=1$.
                
                    \item[--] If $|\mv{v}_{\varrho,i}|<\sigma_i$ we obtain
                    \begin{align*}
                        \frac{\partial}{\partial q_j} \chi_{i,l} = \m \varrho \, \sigma_i \, \delta_{j,L(i-1)+l}, \qquad
                        \frac{\partial}{\partial \mu_j} \chi_{i,l} = 0
                    \end{align*}
                    for all $(i,l,j)\in \ul{N} \times \ul{L} \times \ul{dM}$, which results in
                    \begin{align}\label{eq:strucure_Xi_Yi_C2}
                        \bs{X}_i
                        = - \varrho \, \sigma_i \, \bs{I}_{L\times L}, \qquad
                        \bs{Y}_i
                        = \bs{0}.
                    \end{align}
                \end{itemize}
                Note that \eqref{eq:strucure_Xi_Yi_C2} is of the structure \eqref{eq:desiredStruc_Xi_Yi} with $\tau_i=0$.
                \item[\bf 2.] If there exists an index set $I\subseteq \ul{N}$ such that $|\mv{v}_{\varrho,i}| = \sigma_i$ for $i\in I$, $\F$ is not differentiable in $(\mv{a}, \mv{b}, \mv{c})$ and, hence, by the definition of the Clarke subdifferential any $\bs{H}\in\partial\F(\mv{a},\mv{b},\mv{c})$ takes the form
                \begin{align*}
                    \bs{H}
                     = t \, \lim_{n\to\infty} \nabla\F(\mv{a}_n,\mv{b}_n,\mv{c}_n) + (1-t) \, \lim_{n\to\infty} \nabla\F(\widetilde{\mv{a}}_n,\widetilde{\mv{b}}_n,\widetilde{\mv{c}}_n)
                \end{align*}
                for some $t\in[0,1]$. Here, $(\mv{a}_n,\mv{b}_n,\mv{c}_n), (\widetilde{\mv{a}}_n,\widetilde{\mv{b}}_n,\widetilde{\mv{c}}_n)\to(\mv{a},\mv{b},\mv{c})$ as $n\to\infty$ and $\F$ is differentiable in any point $(\mv{a}_n,\mv{b}_n,\mv{c}_n)$, $(\widetilde{\mv{a}}_n,\widetilde{\mv{b}}_n,\widetilde{\mv{c}}_n)$. Note, the limits only exist if for any $i\in\ul{N}$
                \begin{align*}
                    \vert \mv{v}_{\varrho,i}^{(n)}\vert > \sigma_i \quad \text{or} \quad 
                    \vert \mv{v}_{\varrho,i}^{(n)}\vert < \sigma_i, \qquad
                    \vert \widetilde{\mv{v}}_{\varrho,i}^{(n)}\vert > \sigma_i \quad \text{or} \quad 
                    \vert \widetilde{\mv{v}}_{\varrho,i}^{(n)}\vert < \sigma_i
                \end{align*}
                holds true for all but finitely many $n$. Additionally, for $i\in I$, we have $|\mv{v}_{\varrho,i}^{(n)}|, |\widetilde{\mv{v}}_{\varrho,i}^{(n)}| \to \sigma_i$ as $n\to\infty$. 
                \begin{itemize}
                    \item[--] If $\vert \mv{v}_{\varrho,i}^{(n)}\vert > \sigma_i$ and $\vert \widetilde{\mv{v}}_{\varrho,i}^{(n)}\vert > \sigma_i$ for all but finitely many $n$ as in \textbf{1.} we find that
                    \begin{align*}
                        \bs{X}_i 
                          &= t \, \bigg( \varrho \, \mv{c}_{i} \otimes \frac{\mv{v}_{\varrho,i}}{|\mv{v}_{\varrho,i}|} - \varrho \, \sigma_i \, \bs{I}_{L\times L} \bigg) 
                         + (1-t) \, \bigg( \varrho \, \mv{c}_{i} \otimes \frac{\mv{v}_{\varrho,i}}{|\mv{v}_{\varrho,i}|} - \varrho \, \sigma_i \, \bs{I}_{L\times L} \bigg) \\
                         &=\varrho \, \mv{c}_{i} \otimes \frac{\mv{v}_{\varrho,i}}{|\mv{v}_{\varrho,i}|} - \varrho \, \sigma_i \, \bs{I}_{L\times L}\\
                        \bs{Y}_i
                          &= t \, \bigg( \mv{c}_{i} \otimes \frac{\mv{v}_{\varrho,i}}{|\mv{v}_{\varrho,i}|} + \big(|\mv{v}_{\varrho,i}|-\sigma_i\big) \, \bs{I}_{L\times L} \bigg) + (1-t) \, \bigg( \mv{c}_{i} \otimes  \frac{\mv{v}_{\varrho,i}}{|\mv{v}_{\varrho,i}|} + \big(|\mv{v}_{\varrho,i}|-\sigma_i\big) \, \bs{I}_{L\times L} \bigg)\\
                          &=\mv{c}_{i} \otimes \frac{\mv{v}_{\varrho,i}}{|\mv{v}_{\varrho,i}|} + \big(|\mv{v}_{\varrho,i}|-\sigma_i\big) \, \bs{I}_{L\times L}
                    \end{align*}
                    for any $i\in\ul{N}$. Thus, for $i\in(\ul{N}\setminus I)$ we obtain the structure \eqref{eq:strucure_Xi_Yi_C1} and for $i\in I$, the formulas simplify to
                    \begin{align*}
                        \bs{X}_i 
                         = \frac{\varrho}{\sigma_i} \, \mv{c}_{i} \otimes \mv{v}_{\varrho,i} - \varrho \, \sigma_i \, \bs{I}_{L\times L}, \qquad
                        \bs{Y}_i
                         = \frac{1}{\sigma_i}\, \mv{c}_{i} \otimes \mv{v}_{\varrho,i}.
                    \end{align*}
                    \item[--] If $\vert \mv{v}_{\varrho,i}^{(n)}\vert < \sigma_i$ and $\vert \widetilde{\mv{v}}_{\varrho,i}^{(n)}\vert < \sigma_i$ for all but finitely many $n$ as in \textbf{1.} we find that
                    \begin{align*}
                        \bs{X}_i 
                          &= t \, \big( - \varrho \, \sigma_i \, \bs{I}_{L\times L} \big) 
                         + (1-t) \, \big( - \varrho \, \sigma_i \, \bs{I}_{L\times L} \big)
                          = - \varrho \, \sigma_i \, \bs{I}_{L\times L} \\
                        \bs{Y}_i
                          &= t \, \bs{0} + (1-t) \, \bs{0}
                          = \bs{0}
                    \end{align*}
                    for any $i\in\ul{N}$, which is \eqref{eq:desiredStruc_Xi_Yi} with $\tau_i = 0$.
                    \item[--] If $\vert \mv{v}_{\varrho,i}^{(n)}\vert > \sigma_i$ and $\vert \widetilde{\mv{v}}_{\varrho,i}^{(n)}\vert < \sigma_i$ for all but finitely many $n$ we find that
                    \begin{align*}
                        \bs{X}_i 
                          &= t \, \bigg( \varrho \, \mv{c}_{i} \otimes \frac{\mv{v}_{\varrho,i}}{|\mv{v}_{\varrho,i}|} - \varrho \, \sigma_i \, \bs{I}_{L\times L} \bigg)
                           + (1-t) \, \big( - \varrho \, \sigma_i \, \bs{I}_{L\times L} \big) \\
                          & \qquad = t \, \varrho \, \mv{c}_{i} \otimes \frac{\mv{v}_{\varrho,i}}{|\mv{v}_{\varrho,i}|} - \varrho \, \sigma_i \, \bs{I}_{L\times L} \\
                        \bs{Y}_i
                         &= t \, \bigg( \mv{c}_{i} \otimes \frac{\mv{v}_{\varrho,i}}{|\mv{v}_{\varrho,i}|} + \big(|\mv{v}_{\varrho,i}|-\sigma_i\big) \, \bs{I}_{L\times L} \bigg)
                          + (1-t) \, \bs{0} \\
                         & \qquad = t \, \bigg( \mv{c}_{i} \otimes \frac{\mv{v}_{\varrho,i}}{|\mv{v}_{\varrho,i}|} + \big(|\mv{v}_{\varrho,i}|-\sigma_i\big) \, \bs{I}_{L\times L} \bigg)
                    \end{align*}
                    for any $i\in\ul{N}$ which is \eqref{eq:desiredStruc_Xi_Yi} with $\tau_i = t$. For $i\in I$ the formulas simplify to
                    \begin{align*}
                        \bs{X}_i 
                         = \frac{t\, \varrho}{\sigma_i} \, \mv{c}_{i} \otimes \mv{v}_{\varrho,i} - \varrho \, \sigma_i \, \bs{I}_{L\times L}, \qquad
                        \bs{Y}_i
                         = \frac{t}{\sigma_i}\, \mv{c}_{i} \otimes \mv{v}_{\varrho,i}.
                    \end{align*}
                    \item[--] If $\vert \mv{v}_{\varrho,i}^{(n)}\vert < \sigma_i$ and $\vert \widetilde{\mv{v}}_{\varrho,i}^{(n)}\vert > \sigma_i$ for all but finitely many $n$ we find that
                    \begin{align*}
                        \bs{X}_i
                         &= t \, \big( - \varrho \, \sigma_i \, \bs{I}_{L\times L} \big)
                          + (1-t) \, \bigg( \varrho \, \mv{c}_{i} \otimes \frac{\mv{v}_{\varrho,i}}{|\mv{v}_{\varrho,i}|} - \varrho \, \sigma_i \, \bs{I}_{L\times L} \bigg) \\
                         & \qquad = (1-t) \, \varrho \, \mv{c}_{i} \otimes \frac{\mv{v}_{\varrho,i}}{|\mv{v}_{\varrho,i}|} - \varrho \, \sigma_i \, \bs{I}_{L\times L} \\
                        \bs{Y}_i
                         &= t \, \bs{0} + (1-t) \, \bigg( \mv{c}_{i} \otimes \frac{\mv{v}_{\varrho,i}}{|\mv{v}_{\varrho,i}|} + \big(|\mv{v}_{\varrho,i}|-\sigma_i\big) \, \bs{I}_{L\times L} \bigg) \\
                         & \qquad = (1-t) \, \bigg( \mv{c}_{i} \otimes \frac{\mv{v}_{\varrho,i}}{|\mv{v}_{\varrho,i}|} + \big(|\mv{v}_{\varrho,i}|-\sigma_i\big) \, \bs{I}_{L\times L} \bigg)
                    \end{align*}
                    for any $i\in\ul{N}$ which is \eqref{eq:desiredStruc_Xi_Yi} with $\tau_i = 1 - t$. For $i\in I$ the formulas simplify to
                    \begin{align*}
                        \bs{X}_i 
                         = \frac{(1-t) \, \varrho}{\sigma_i} \, \mv{c}_{i} \otimes \mv{v}_{\varrho,i} - \varrho \, \sigma_i \, \bs{I}_{L\times L}, \qquad
                        \bs{Y}_i
                         = \frac{1 - t}{\sigma_i}\, \mv{c}_{i} \otimes \mv{v}_{\varrho,i}.
                    \end{align*}
                \end{itemize}
            \end{enumerate}
            This completes the proof.
        \end{proof}

        \begin{lemma}\label{lem: Elasto X Y Preliminaries}
            Let $\bs{H}^*\in\partial\F(\mv{a}^\ast,\mv{b}^\ast,\mv{c}^\ast)$ with the corresponding matrices $\bs{X}_i$ and $\bs{Y}_i$ of the form \eqref{eq:desiredStruc_Xi_Yi}. Then, the pairs $(\bs{X}_i,\bs{Y}_i)$ are eigencomplementary.
        \end{lemma}
        \begin{proof}
            %
            Recall, that $\mv{a}^{\ast}\in\RR^{dM}$, $\mv{b}^\ast\in\RR^{LN}$ and $\mv{c}^\ast = (\lambda_1,\ldots,\lambda_{LN})^{\top}\in\RR^{LN}$ are the coefficient vectors corresponding to the discrete solution $(\mv{u}_{hp},\bs{p}_{hp},\bs{\lambda}_{hp})\in\VV_{hp}\times\Lambda_{hp}$ of \eqref{eq:discrete_mixed_variationalF} and
            \begin{align*}
                \bs{p}_{hp} = \sum_{i\in\ul{N}} \bs{p}_i \, \phi_i, \qquad
                \bs{\lambda}_{hp} = \sum_{i\in\ul{N}} \bs{\lambda}_i \, \varphi_i.
            \end{align*}
            Moreover, for any $i\in\ul{N}$, it holds that
            \begin{align*}
                \bs{\lambda}_i = \sum_{k\in\ul{L}} \lambda_{L(i-1)+k} \, \bs{\Phi}_k
                 \qquad\Longrightarrow\qquad
                |\bs{\lambda_i}|_F = \bigg( \sum_{k\in\ul{L}} \lambda_{L(i-1)+k}^2 \bigg)^{1/2}=|\mv{c}_{i}^\ast|
            \end{align*}
            due to the fact that $\bs{\Phi}_1,\dots, \bs{\Phi}_L$ form an orthonormal basis of $\SS_{d,0}$ with respect to the Frobenius inner product. From \eqref{eq: xi 0 condition} and \eqref{eq: Chi Basisdarstellung} it follows that
            \begin{align*}
                \bs{\chi}_i(\bs{p}_i,\bs{\lambda}_i) = \sum_{k\in\ul{L}} \chi_{i,k} \, \bs{\Phi}_k = \bs{0}
            \end{align*}
            and, thus,
            \begin{align}\label{eq:wichtige_Gleichh}
                \max \big\lbrace \sigma_i, \vert \mv{v}_{\varrho,i}^{\ast} \vert \big\rbrace \, \mv{c}_i^{\ast} = \sigma_i \, \mv{v}_{\varrho,i}^{\ast}
            \end{align}
            by \eqref{eq:component_chi} and recalling that $\mv{v}_{\varrho,i}^{\ast} = \mv{c}_{i}^\ast + \varrho \, \mv{b}_{i}^\ast$. By Lemma~\ref{lem:structure_subgradient} the matrices $\bs{X}_i, \bs{Y}_i\in\RR^{L\times L}$ take the form
            \begin{align}\label{eq:allgDarst_Xi_Yi}
                \bs{X}_i
                 = \tau_i\, \varrho \, \mv{c}_{i}^\ast \otimes \frac{\mv{v}_{\varrho,i}^\ast}{|\mv{v}_{\varrho,i}^\ast|} - \varrho \, \sigma_i \, \bs{I}_{L\times L}, \qquad
                \bs{Y}_i
                 = \tau_i\left(\mv{c}_{i}^\ast \otimes  \frac{\mv{v}_{\varrho,i}^\ast}{|\mv{v}_{\varrho,i}^\ast|} + \big(|\mv{v}_{\varrho,i}^\ast|-\sigma_i\big) \, \bs{I}_{L\times L}\right)
            \end{align}
            for arbitrary $i\in\ul{N}$ with $\tau_i\in[0,1]$. We distinguish two cases:
            \begin{enumerate}
                \item[\bf 1.] If $|\mv{v}_{\varrho,i}^\ast| \geq \sigma_i$, by \eqref{eq:wichtige_Gleichh} it follows that
                \begin{align}\label{eq:interessante_Impl}
                    |\mv{v}_{\varrho,i}^\ast| \, \mv{c}_i^\ast = \sigma_i \, \mv{v}_{\varrho,i}^\ast
                     \qquad\Longrightarrow\qquad
                     \frac{\mv{v}_{\varrho,i}^\ast}{|\mv{v}_{\varrho,i}^\ast|}
                     = \frac{1}{\sigma_i} \, \mv{c}_i^\ast.
                \end{align}
                Thus, \eqref{eq:allgDarst_Xi_Yi} reduces to
                \begin{align*}
                    \bs{X}_i
                     = \frac{\tau_i\, \varrho}{\sigma_i} \, \mv{c}_{i}^\ast \otimes \mv{c}_{i}^\ast - \varrho \, \sigma_i \, \bs{I}_{L\times L}, \qquad
                    \bs{Y}_i
                     = \tau_i\bigg( \frac{1}{\sigma_i} \, \mv{c}_{i}^\ast \otimes \mv{c}_{i}^\ast + \big(|\mv{v}_{\varrho,i}^\ast|-\sigma_i\big) \, \bs{I}_{L\times L} \bigg),
                \end{align*}
                from which we obviously obtain the symmetry of both matrices. Moreover, by \eqref{eq:interessante_Impl} we have
                \begin{align*}
                    \abs{\mv{c}_i^\ast} = \sigma_i > 0,
                \end{align*}
                which implies that there is at least one $s\in\ul{L}$ with $\lambda_{L(i-1)+s}\neq 0$. We define
                \begin{align*}
                    \xi_k :=
                    \begin{cases}
                        \m \varrho \, \sigma_i,       & \text{for } k\in \ul{L}\setminus\lbrace s\rbrace, \\
                        (\tau_i - 1) \, \varrho \, \sigma_i, & \text{for } k = s,
                    \end{cases} \qquad
                    \eta_k :=
                    \begin{cases}
                       \tau_i\,\big( \vert\mv{v}_{\varrho,i}^{\ast}\vert - \sigma_i\big),            & \text{for } k\in \ul{L}\setminus\lbrace s\rbrace, \\
                        \tau_i \, \vert \mv{v}_{\varrho,i}^{\ast}\vert, & \text{for } k = s
                    \end{cases}
                \end{align*}
                and choose
                \begin{align*}
                    \mv{v}_s
                     := \frac{1}{\lambda_{L(i-1)+s}} \, \mv{c}_i^\ast, \qquad
                    \mv{v}_k
                     := \mv{e}_k - \frac{\lambda_{L(i-1)+k}}{\lambda_{L(i-1)+s}} \, \mv{e}_s \quad
                     \text{for } k\in\ul{L}\setminus\lbrace s\rbrace,
                \end{align*}
                where $\mv{e}_1,\ldots,\mv{e}_L$ denote the Euclidean unit vectors of $\RR^L$. 
                We now consider
                \begin{align*}
                    \big( \bs{X}_i - \xi_k \, \bs{I}_{L\times L}\big) \, \mv{v}_k 
                     &=  \frac{\tau_i \, \varrho}{\sigma_i} \, \big( \mv{v}_k^{\top} \, \mv{c}_i^\ast \big) \, \mv{c}_i^{\ast} - (\varrho \, \sigma_i + \xi_k) \, \mv{v}_k,\\
                     \big(\bs{Y}_i-\eta_k\,\bs{I}_{L\times L}\big)\,\mv{v}_k &=  \frac{\tau_i}{\sigma_i} \,\big( \mv{v}_k^{\top} \, \mv{c}_i^\ast\big) \, \mv{c}_i^\ast  + \big(\tau_i(|\mv{v}_{\varrho,i}^\ast|-\sigma_i)-\eta_k\big) \,\mv{v}_k .
                \end{align*}
                For $k\in \ul{L}\setminus\lbrace s\rbrace$ it holds that
                    \begin{align*}
                        \mv{v}_k^{\top} \, \mv{c}_i^\ast = 0, \qquad
                         (\varrho \, \sigma_i + \xi_k) \, \mv{v}_k = \mv{0}, \qquad \big(\tau_i(|\mv{v}_{\varrho,i}^\ast|-\sigma_i)-\eta_k\big) \,\mv{v}_k=\mv{o}
                    \end{align*}
                and for $k = s$ it holds that
                    \begin{align*}
                        \mv{v}_s^{\top} \, \mv{c}_i^\ast
                         &= \frac{\vert \mv{c}_i^{\ast}\vert^2}{\lambda_{L(i-1)+s}}
                         = \frac{\sigma_i^2}{\lambda_{L(i-1)+s}},\\
                         (\varrho \, \sigma_i + \xi_s) \, \mv{v}_s
                         &= \frac{\tau_i \, \varrho \, \sigma_i}{\lambda_{L(i-1)+s}} \, \mv{c}_i^\ast,\\
                         \big(\tau_i(|\mv{v}_{\varrho,i}^\ast|-\sigma_i)-\eta_s\big)\,\mv{v}_s &=-\frac{\tau_i\,\sigma_i}{\lambda_{L(i-1)+s}}\,\mv{c}_i^\ast.
                    \end{align*}
                Thus, in total we obtain
                \begin{align*}
                     \big( \bs{X}_i - \xi_k \, \bs{I}_{L\times L}\big) \, \mv{v}_k = \mv{0},\qquad \big(\bs{Y}_i-\eta_k\,\bs{I}_{L\times L}\big)\,\mv{v}_k = \mv{o}
                \end{align*}
                for all $k\in\ul{L}$. Thus, $\bs{X}_i$ and $\bs{Y}_i$ have the the eigenbasis $\mv{v}_1,\ldots,\mv{v}_L$ in common and $\xi_k\leq 0 \leq  \eta_k$. Finally, if both matrices are singular, which means that $\tau_i = 1$ and $\vert\mv{v}_{\varrho,i}^{\ast}\vert = \sigma_i$, then it holds that
                \begin{align*}
                    \Eig_{\bs{X}_i}(-\varrho\sigma_i) = \Eig_{\bs{Y}_i}(0).
                \end{align*}
                This proves the eigencomplementarity of $(\bs{X}_i,\bs{Y}_i)$ for $|\mv{v}_{\varrho,i}^\ast| \geq \sigma_i$.

                \item[\bf 2.] If $|\mv{v}_{\varrho,i}^\ast| < \sigma_i$, then Lemma~\ref{lem:structure_subgradient} yields that $\tau_i =0$, which means that
                \begin{align*}
                    \bs{X}_i 
                      = - \varrho \, \sigma_i \, \bs{I}_{L\times L}, \qquad
                    \bs{Y}_i
                     = \bs{0}.
                \end{align*}
                Both matrices are obviously symmetric, all eigenvalues of $\bs{X}_i$ have the same value $-\varrho\sigma_i< 0$ and all eigenvalues of $\bs{Y}_i$ are zero. They also have the eigenbasis consisting of Euclidean unit vectors of $\RR^L$ in common. Therefore, the pair $(\bs{X}_i,\bs{Y}_i)$ are eigencomplementary by  Definition~\ref{subsec:Eigencomplement}.
            \end{enumerate}
            This completes the argument.
        \end{proof}
        Due to the properties of $\F$ and $H^*\in\F(\mv{a}^\ast,\mv{b}^\ast,\mv{c}^\ast) $ applying Lemma~\ref{lem:structure_subgradient} and Lemma~\ref{lem: Elasto X Y Preliminaries} together with Theorem~\ref{thm:conditions_convergence} immediately leads to the following result:
        \begin{theorem}\label{cor:elastoNewton_Convergence}
            Algorithm \ref{eq:SSN_iterates} is well-defined in a neighborhood of the solution $(\mv{a}^\ast,\mv{b}^\ast,\mv{c}^\ast)^\top\in\RR^{dM}\times\RR^{LN}\times\RR^{LN}$ of \eqref{eq:semismooth_newtonF_Elastoplast_02} and converges locally superlinear.
        \end{theorem}
        The convergence rates of the numerical experiments in the following section show that the order of convergence is greater than one (namely 4/3). This shows that $\F$ might be $\alpha$-order semismooth for some $\alpha\in(0,1)$, which, in view of Theorem~\ref{thm:conditions_convergence}, suggests a convergence of order $1+\alpha$.


\smallskip

\section{Numerical Results}\label{sec:numeric}

Let $\Omega := (\m 1,1)^2$ with Dirichlet boundary $\Gamma_D := [\m 1,1]\times\lbrace \m 1 \rbrace$ and Neumann boundary $\Gamma_N := \partial \Omega \setminus \Gamma_D$. The volume and Neumann forces are defined as $\mv{f} := \mv{0}$ and $\mv{g} := (0,-400 \min(0,x_1^2-1/4)^2)^{\top}$ on $[\m 1,1]\times \lbrace 1\rbrace$ and zero elsewhere on the Neumann boundary, respectively. Furthermore, the material is described by $\myspace{C}\bs{\tau} := \lambda \operatorname{tr}(\bs{\tau})\bs{I}+2\mu\bs{\tau}$ with Lam\'e constants $\lambda:=\mu:=1000$, $\myspace{H}\bs{\tau}:=500\bs{\tau}$ and $\sigma_y:=5$.
In the following, we consider three discretizations:  
\begin{enumerate}[(i)]
    \item A uniform $h$-version with $h=2^{\m 8}$, $p=1$, which leads to a total of $dM=525.312$ degrees of freedom for the displacement variable $\mv{u}_{hp}$ and  $LN=524.288$ degrees of freedom for the plasticity variable $\bs{p}_{hp}$ as well as the discrete Lagrange multiplier $\bs{\lambda}_{hp}$.
    \item A uniform $p$-version with $h=2^{\m 1}$, $p=25$, which leads to $dM=20.200$ and $LN=20.000$ degrees of freedom, respectively.
    \item A discretization resulting from an adaptive $hp$-scheme with 
    $$
        h_{\min}=2^{\m 25}, \ldots , h_{\max}=2^{\m 2},\qquad p_{\min}=1, \ldots ,p_{\max}=7,
    $$
    which leads to $dM=98.572$ and $LN=104.292$ degrees of freedom, respectively; see \cite{ref:Bammer2023Apriori} for more details.
\end{enumerate}

Thereby, we solve \eqref{eq:semismooth_newtonF_Elastoplast_02} with Algorithm~\ref{eq:SSN_iterates} and zero initial solution. Moreover, we choose the projection parameter $\varrho$ in the definition of the functions $\boldsymbol{\chi}_i$, cf.~\eqref{eq:defi_chi_i}, to be $\varrho:=25$. Numerical experiments indicate that the number of semismooth Newton iterations is robust to changes in $\varrho$ and the finite element spaces but of course depends on the problem at hand. Only for rather small $\varrho$ the number of iterations increases notably. For large $\varrho$ the local convergence radius decreases, which necessitates the use of an adequate back tracking strategy.
We take the 10th iterate with a merit value $\frac12|\F|^2$  of $2.38 \cdot 10^{\m 24}$ ($h$-version) or of $9.37\cdot 10^{\m 23}$ ($p$-version) as a substitute for the exact solution $(\mv{a}^{\ast},\mv{b}^{\ast},\mv{c}^{\ast})$ for the uniform cases. For the $hp$-adaptive scheme, it is the 11th iterate with a merit value of $2.02 \cdot 10^{\m 23}$. In Figure~\ref{fig:SSN_speed} we plot the quotients
\begin{align*}
     q_\text{SSN}(\rho):= \frac{\left| \left( \mv{a}^{(k+1)}-\mv{a}^{\ast},\mv{b}^{(k+1)}-\mv{b}^{\ast},\mv{c}^{(k+1)}-\mv{c}^{\ast} \right) \right| }{ \left| \left( \mv{a}^{(k)}-\mv{a}^{\ast},\mv{b}^{(k)}-\mv{b}^{\ast},\mv{c}^{(k)}-\mv{c}^{\ast} \right) \right|^\rho}, \qquad k=0,1,2,\ldots
\end{align*}
for $\rho \in \{1,4/3,2\}$ and the three discretizations to numerically validate the convergence order of the semismooth Newton method. As $q_\text{SSN}(1) \rightarrow 0$ but $q_\text{SSN}(2) \rightarrow \infty$ Algorithm \ref{eq:SSN_iterates} converges locally superlinear in accordance to Theorem~\ref{cor:elastoNewton_Convergence} but not quadratically. Moreover, we find that the semismooth Newton method exhibits a $4/3$-order convergence as $q_\text{SSN}(4/3)$ is bounded; see Figure~\ref{fig:qssn_4_3}.

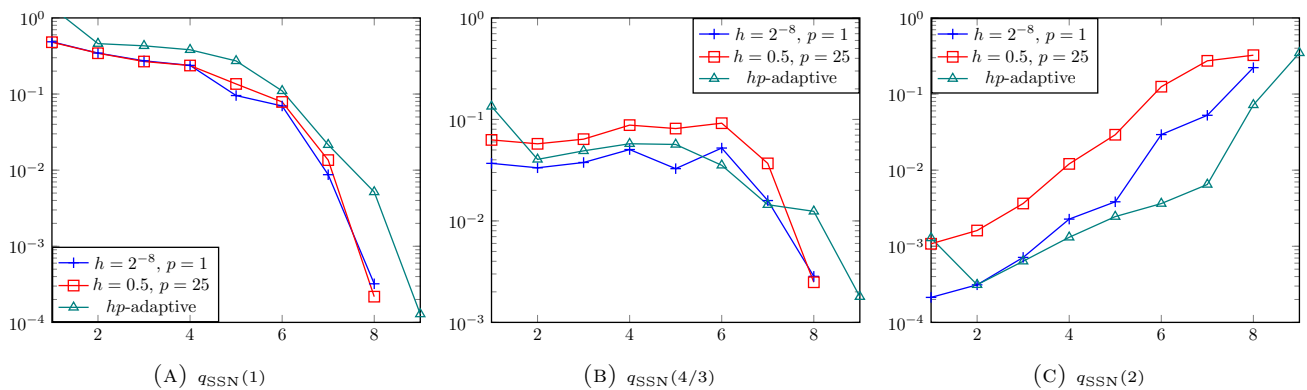
\begin{figure}[htb]
  \centering\hspace{-1em}
\subfloat[\tiny $q_\text{SSN}(1)$]{
	\begin{tikzpicture}[scale=0.66]
		\begin{semilogyaxis}[
			width=0.525\textwidth,
			mark size=3pt,
			line width=0.75pt,
			xmin=1,xmax=9,
			ymin=1e-4,ymax=1,
			legend style={at={(0,0)},anchor=south west},
            legend style={fill=none}
			]
			\addplot+[mark=+, color=blue] table[x index=0,y index=1] {data_ssn_h1.txt};
 			\addplot+[mark=square, color=red] table[x index=0,y index=1] {data_ssn_p.txt};
 			\addplot+[mark=triangle, color=teal] table[x index=0,y index=1] {data_ssn_hp.txt};

 			\legend{{$h=2^{-8}$, $p=1$},{$h=0.5$, $p=25$},{$hp$-adaptive}}
		\end{semilogyaxis}
	\end{tikzpicture}}
	\hspace{0.0cm}  \subfloat[\tiny $q_\text{SSN}(4/3)$]{\label{fig:qssn_4_3}
	\begin{tikzpicture}[scale=0.66]
		\begin{semilogyaxis}[
			width=0.525\textwidth,
			mark size=3pt,
			line width=0.75pt,
			xmin=1,xmax=9,
			ymin=1e-3,ymax=1,
			legend style={at={(1,1)},anchor=north east},
            legend style={fill=none}
			]
			\addplot+[mark=+, color=blue] table[x index=0,y index=2] {data_ssn_h1.txt};
 			\addplot+[mark=square, color=red] table[x index=0,y index=2] {data_ssn_p.txt};
            \addplot+[mark=triangle, color=teal] table[x index=0,y index=2] {data_ssn_hp.txt};

 			\legend{{$h=2^{-8}$, $p=1$},{$h=0.5$, $p=25$},{$hp$-adaptive}}
		\end{semilogyaxis}
	\end{tikzpicture}}
	\hspace{0.0cm}  \subfloat[\tiny $q_\text{SSN}(2)$]{
	\begin{tikzpicture}[scale=0.66]
		\begin{semilogyaxis}[
			width=0.525\textwidth,
			mark size=3pt,
			line width=0.75pt,
			xmin=1,xmax=9,
			ymin=1e-4,ymax=1,
			legend style={at={(0,1)},anchor=north west},
            legend style={fill=none}
			]
			\addplot+[mark=+, color=blue] table[x index=0,y index=3] {data_ssn_h1.txt};
 			\addplot+[mark=square, color=red] table[x index=0,y index=3] {data_ssn_p.txt};
            \addplot+[mark=triangle, color=teal] table[x index=0,y index=3] {data_ssn_hp.txt};

 			\legend{{$h=2^{-8}$, $p=1$},{$h=0.5$, $p=25$},{$hp$-adaptive}}
		\end{semilogyaxis}
	\end{tikzpicture}}
	\vspace{-0.25cm}
	\caption{\small\emph{Iteration number $k$ vs.~$q_\text{SSN}(\rho)$.}} \label{fig:SSN_speed}
\end{figure}


\smallskip

\bibliographystyle{amsalpha}

\smallskip

\begin{thebibliography}{9}


\bibitem{ref:Bammer2022Icosahom}
P.~Bammer, L.~Banz and A.~Schr\"{o}der, $hp$-Finite Elements with Decoupled Constraints for Elastoplasticity, \emph{Spectral and High Order Methods for Partial Differential Equations ICOSAHOM 2020+ 1, Springer} (2023) 141-153.

\bibitem{ref:Bammer2023Apriori} 
P.~Bammer, L.~Banz and A.~Schr\"{o}der, Mixed finite elements of higher-order in elastoplasticity, \emph{Appl. Numer. Math.} \textbf{209} (2025) 38-54.

\bibitem{ref:Bammer2023Posteriori}
P.~Bammer, L.~Banz and A.~Schr\"{o}der, A posteriori error estimates for hp-FE discretizations in elastoplasticity, \emph{Comput. Math. Appl.} \textbf{200} (2025) 448-466.


\bibitem{ref:Banz2015Tresca}
L.~Banz and E.~Stephan, On $hp$-adaptive BEM for frictional contact problems in linear elasticity, \emph{Comput. Math. Appl.} \textbf{69} (2015) 559-581.



\bibitem{ref:Chen1988}
W.F.~Chen and D.J.~Han, \textit{Plasticity for Structural Engineers}, Springer, 1st edition, 1988.


\bibitem{DeLuca1996}
T.~De~Luca and F.~Facchinei and C.~Kanzow, A Semismooth Equation Approach to the Solution of Nonlinear Complementarity Problems, \emph{Math. Program.} \textbf{75} (1996) 407--439.


\bibitem{ref:Gwinner2009}
J.~Gwinner, On the $p$-version approximation in the boundary element method for a variational inequality of the second kind modelling unilateral contact and given friction, \emph{Appl.~Numer.~Math.} \textbf{59} (2009) 2774--1784.


\bibitem{ref:Han1995}
W.~Han and B.D.~Reddy, On the Finite Element Method for Mixed Variational Inequalities Arising in Elastoplasticity, \emph{SIAM J. Numer. Anal.} \textbf{32} (1995) 1778--1807.

\bibitem{ref:Han2013} 
W.~Han and B.D.~Reddy, \textit{Plasticity. Mathematical Theory and Numerical Analysis}, Springer, 2 edition, 2013.

\bibitem{ref:Han2002}
W.~Han and M.~Sofonea, \textit{Quasistatic contact problems in viscoelasticity and viscoplasticity}, AMS, 2002.

\bibitem{ref:Haslinger1999}
J.~Haslinger and M.~Miettinen and P.D.~Panagiotopoulos, \textit{Finite Element Method for Hemivariational Inequalities}, Springer, 1st edition, 1999.

\bibitem{Hintermuller2002}
M.~Hintermüller and K.~Ito and K.~Kunisch, The {{Primal-Dual Active Set Strategy}} as a {{Semismooth Newton Method}}, \emph{SIAM J. Optim.} \textbf{13} (2002) 865--888.

\bibitem{ref:Homand2000}
S.~Homand and J.F.~ Shao, Mechanical behaviour of a porous chalk and effect of saturating fluid, \emph{Mech. Cohes.-Frict. Mater.} \textbf{5} (2000) 583-606.

\bibitem{Hueber2008}
S.~Hüeber and G.~Stadler and B.~I.~Wohlmuth, A {{Primal-Dual Active Set Algorithm}} for {{Three-Dimensional Contact Problems}} with {{Coulomb Friction}}, \emph{SIAM J. Sci. Comput.} \textbf{30} (2008) 572--596.

\bibitem{ref:Josefson1995}
B.L.~Josefson and U.~Stigh and H.E.~Hjelm, A Nonlinear Kinematic Hardening Model for Elastoplastic Deformations in Grey Cast Iron, \emph{J. Eng. Mater. Technol.} \textbf{117} (1995) 145-150.


\bibitem{ref:Kikuchi1988}
N.~Kikuchi and J.T.~Oden, \textit{Contact Problems in Elasticity}, SIAM, 1988. 


\bibitem{ref:Lamichhane2007biorthogonal}
B.~Lamichhane and B.~Wohlmuth, Biorthogonal bases with local support and approximation properties, \emph{Math. Comput.} \textbf{76} (2007) 233-249.




\bibitem{ref:Melenk2005}
J.M.~Melenk, $hp$-interpolation of nonsmooth functions and an application to $hp$-a posteriori error estimation, \emph{SIAM J~ Numer.~Anal.} \textbf{43} (2005) 127--155.

\bibitem{ref:Meyer2018}
K.A.~Meyer and M.~Ekh and J.~Ahlstr\"{o}m, Modeling of kinematic hardening at large biaxial deformations in pearlitic rail steel, \emph{Int J Solids Struct} \textbf{130-131} (2018) 122-132.


\bibitem{ref:Poltronieri2014}
F.~Poltronieri and A.~Piccolroaz and D.~Bigoni and S. Romero Baivier, A simple and robust elastoplastic constitutive model for concrete, \emph{Eng. Struct.} \textbf{60} (2014) 81-84.


\bibitem{ref:Qi1998}
L.~Qi and D. Sun, Nonsmooth Equations and Smoothing {{Newton}} Methods, \emph{Appl. Optimizat.} \textbf{98} (1998) 1--22.

\bibitem{ref:Qi1993}
L.~Qi and J.~Sun, A nonsmooth version of Newton’s method, \emph{Math.~Program.} \textbf{58} (1993) 353–-367.


\bibitem{ref:Schroeder_PAMM2011} 
A.~Schr\"{o}der. Mixed FEM of higher‐order for a frictional contact problem. \emph{PAMM} \textbf{11} (2011) 7-10.

\bibitem{ref:Schroeder_MFE2011}
A.~Schr\"{o}der. Mixed finite element methods of higher-order for model contact problems. \emph{SIAM J.~Numer.~Anal.} \textbf{49}(6) (2011) 2323–2339.

\bibitem{ref:Schroeder_aPost2012}
A.~Schr\"{o}der. A posteriori error estimates of higher-order finite elements for frictional contact problems. \emph{Comput.~Methods in Appl.~Mech.~Eng.} \textbf{249} (2012) 151–157.

\bibitem{ref:Schroeder_contact2011}
A.~Schr\"{o}der, H.~Blum, A.~Rademacher and H.~Kleemann. Mixed fem of higher order for contact problems with friction. \emph{Int.~J.~Numer.~Anal.~Model.} \textbf{8}(2) (2011) 302–323.

\bibitem{ref:Schroeder2011}
A.~Schr\"{o}der and S.~Wiedemann, Error estimates in elastoplasticity using a mixed method, \emph{Appl.~Numer.~Math.} \textbf{61} (2011) 1031--1045.



\bibitem{ref:Wriggers2006}
P.~Wriggers and T.A.~Laursen, Computational Contact Mechanics, Springer, 2006.

\bibitem{ref:Wong2006}
H.~Wong and C.J.~Leo, A simple elastoplastic hardening constitutive model for EPS geofoam, \emph{Geotext. Geomembr.} \textbf{24} (2006) 299-310.

\bibitem{ref:Zhang2005}
F.~Zhang, \textit{The Schur Complement and Its Applications}, Springer, 1st edition, 2005.

\end{thebibliography}

\end{document}